\documentclass[psamsfonts]{amsart}

\usepackage{amssymb,amsfonts,amsthm}
\usepackage{amsmath}
\usepackage[all,arc]{xy}
\usepackage{enumerate}
\usepackage{mathrsfs}
\usepackage{comment}

\newcommand{\CH}{\mathcal{H}}

\newcommand{\CalW}{\mathcal{W}}
\newcommand{\CalA}{\mathcal{A}}
\newcommand{\CalL}{\mathcal{L}}

\newcommand{\CalO}{\mathcal{O}}

\newcommand{\eps}{{\varepsilon}}

\newcommand{\C}{{\mathbb C}}

\newcommand{\RR}{{\mathbb R}}

\newcommand{\NN}{{\mathbb N}}
\newcommand{\ZZ}{{\mathbb Z}}
\newcommand{\TT}{{\mathbb T}}
\newcommand{\CC}{{\mathbb C}}

\newcommand{\EE }{{\mathbb E}}

\newcommand{\II}{{\mathbb I}}
\newcommand{\DD}{{\mathbb D}}

\newcommand{\QQ}{{\mathbb Q}}

\newcommand{\cA}{{\CalA}}

\newcommand\cD{{\mathcal  D}}

\newcommand\cO{{\CalO}}

\newcommand{\Ker}{\text{Ker}}

\newcommand{\at}{{\alpha_t }}
\renewcommand{\ar}{{\alpha_r }}

\newcommand{\ccU}{\check U}

\newcommand{\uG}{\underline G}

\newcommand{\chapter}{\part}

\def\eps{\epsilon }

\def\D{\partial }
\newcommand\adots{\mathinner{\mkern2mu\raise1pt\hbox{.}
\mkern3mu\raise4pt\hbox{.}\mkern1mu\raise7pt\hbox{.}}}

\newtheorem{thm}{Theorem}[section]

\newtheorem{prop}[thm]{Proposition}
\newtheorem{lem}[thm]{Lemma}

\newtheorem{assump}[thm]{Assumption}

\theoremstyle{definition}
\newtheorem{defn}[thm]{Definition}

\newtheorem{notn}[thm]{Notation}

\theoremstyle{remark}
\newtheorem{rem}[thm]{Remark}

\renewenvironment{align}{
    \begin{equation}
    \begin{aligned}
	}
	{
    \end{aligned}
    \end{equation}
    \ignorespacesafterend
}

\makeatletter
\let\c@equation\c@thm
\makeatother
\numberwithin{equation}{section}

\title[Hyperbolic Boundary Value Problems]{First Order Hyperbolic Boundary Value Problems with a Large Oscillatory Zero Order Term}

\author{Zhaodh, Alvis}

\date{\today}

\calclayout
\begin{document}

\begin{titlepage}
	\centering
	{\huge\bfseries First Order Hyperbolic Boundary Value Problems with a Large Oscillatory Zero Order Term\par}
	\vspace{2cm}

	{\Large\itshape Zhaodh, Alvis \par}
		\vspace{1cm}
		{\scshape University of North Carolina at Chapel Hill \par}

	{\scshape Senior Honors Thesis (Department of Mathematics)\par}
			\vspace{1.5cm}
			{\itshape May 4th, 2021 \par}

	\vfill

	\begin{equation*}
	\begin{aligned}
	\hspace{4in} \text{Approved by}&\\
	&\text{Williams, Mark (Thesis Advisor)}\\
	&\text{Taylor, Michael (Reader)}\\
	&\text{Metcalfe, Jason (Reader)}
	\end{aligned}
	\end{equation*}

	\vfill
	
This project was supported by the Wilson Family Honors Excellence Fund administered by Honors Carolina.
\end{titlepage}

\begin{abstract}
We study the weakly stable hyperbolic boundary value problem with a large zero order oscillatory coefficient. This problem is related to linearized problems in the study of Mach stem and vortex sheets.  We wish to establish a uniform energy estimate with respect to $\eps$, which is needed in mentioned applications and justification of geometric optics solutions,  but the zero order oscillatory term gives rise to great obstacles.  In this paper we obtain positive results in the small/medium frequency region by adapting the approach of \cite{W4} to a more general situation than the one treated there.  We also show that it is possible to construct high order approximate solutions by the method of geometric optics for those systems without any restriction on frequencies.
\end{abstract}

\maketitle
\enlargethispage{5\baselineskip}
\tableofcontents
\newpage

\section{Introduction}
Consider the following linear hyperbolic system on $\Omega = \RR \times \{(x_1,x_2):x_2 \geq 0\}$: 

\begin{equation} \label{eqn:1system}
	\begin{aligned}
&L(\partial)u + \mathcal{D}\left(\frac{\phi_{in}}{\eps}\right)u:= \partial_tu + B_1\partial_{x_1}u + B_2\partial_{x_2}u + \mathcal{D}\left(\frac{\phi_{in}}{\eps}\right)u = F(t,x,\frac{\phi_0}{\eps}) \text{ in } x_2>0 \\ 
&Bu := \epsilon G \left( t,x_1,\frac{\phi_0}{\epsilon} \right) \text{ on }x_2 = 0 \\ 
&u = 0 \text{ in } t<0 
	\end{aligned}
\end{equation}
where all $B_j$'s are $N \times N$ matrices, $B_2$ is invertible, $\mathcal{D}(\theta_{in}),F(t,x,\theta),G(t,x,\theta)$ is of period $2\pi$  in $\theta_{in}$, and $\theta$ receptively. We note that the oscillatory coefficient  $\mathcal{D}\left(\frac{\phi_{in}}{\eps}\right)$ is large, $O(1)$, compared to the small wave length $\eps \in (0,\eps_0]$.  The problem is weakly stable, that is $(L(\D),B)$ fails to satisfy the uniform Lopatinski condition (Definition \ref{ulcdefn}) in a specific way (Assumption \ref{lopassump}). The boundary matrix $B$ is a constant $p \times N$ matrix with appropriate rank $p$. The boundary phase is $\phi_0(t,x_1) = \beta_l \cdot (t,x_1)$, where $\beta_l = (\sigma_l,\eta_l) \in \RR^2 \setminus 0$ is a direction where uniform Lopatinski condition fails. The interior phase:
\begin{equation}
\phi_{in}(t,x) = \phi_0(t,x_1) + \omega_N(\beta_l)x_2 
\end{equation}
is incoming (Definition \ref{defnincoming}).

 As a warm-up we first establish some techniques in the case (section \ref{section1}) where uniform Lopatinski condition defined in \ref{ulcdefn} holds. For simplicity, we take the oscillatory coefficient to be $e^{i\theta_3}M$ in that section,  where $M$ is a constant $N \times N$ matrix. We construct an approximate WKB solutions (or geometric optics solutions) (\ref{eqn:profile}) and use an existing estimate in \cite{Kre70} to justify this solution by showing that it  is close to the exact solution in $L^\infty$. 

 The main goal of this paper is to study the situation where uniform Lopatinski condition fails in a certain way (\ref{lopassump}) (we say the problem is weakly stable), the oscillatory coefficient has both positive and negative Fourier spectrum, and there are $p \geq 2$ incoming modes ($\phi_N-p+1,... ,\phi_N$) as in definition \ref{defnincoming}.  In this case, methods that have successfully established a uniform estimate for other hyperbolic systems fails to a yield uniform estimate for \eqref{eqn:1system}. Here we refer to methods developed for  (a) Problems where ULC holds \cite{Kre70}, \cite{CGW11}; (b)weakly stable problems with non-oscillatory coefficients \cite{Cou04}, \cite{Cou05}, \cite{CS04}; and (c)weakly stable problems where oscillatory coefficient is small, $O(\eps)$ \cite{CGW14}. 

To further see the difficulties of \eqref{eqn:1system}, we note that when ULC fails, methods established in previous papers mentioned above give an estimate to the system $(L(\D),B)$ (no oscillatory term) in the form:
\begin{equation} \label{lodrtv}
|u|_{L^2} \leq C_T(|L(\D)u|_{H^1}+<Bu>_{H^1})
\end{equation} 
If the oscillatory term is $\eps \mathcal{D}$, \cite{CG10} shows that \eqref{lodrtv} still applies with a constant $C_T$ independent of $\eps\in (0,1]$ when $L(\partial)$ is replaced by $L(\partial)+\eps\mathcal{D}(\phi_{in}/\eps)$.
In \eqref{eqn:1system} the oscillatory term is large (not uniformly Lipschitzean with respect to $\eps\in (0,1]$), so the argument of \cite{CG10} breaks down and fails to establish an estimate like \eqref{lodrtv} when $L(\partial)$ is replaced by $L(\partial)+\mathcal{D}(\phi_{in}/\eps)$.   Thus, new methods are needed to obtain estimates for  \eqref{eqn:1system}.

A different method established in \cite{W4} gives an energy estimate for \eqref{eqn:1system}, which  works if either: (1) the oscillatory term has only positive Fourier spectrum;  or  (2) the oscillatory term has both positive and negative Fourier spectrum and there is only one incoming phase. In particular, theorem 2.12 in \cite{W4} can't be applied to the system \eqref{eqn:1system} when there are at least two incoming phases.

This paper mainly deals with the case when  \eqref{eqn:1system} satisfies the following conditions: (1). The oscillatory term has both positive and negative Fourier spectrum; (2). There are at least two incoming phases,; (3). There is only one direction $\beta_l$ where ULC fails; (4). There is no resonance as defined in definition \ref{defnres}.  The first main result of this thesis is the energy estimate of theorem \ref{mainthm} for systems of the form  \eqref{eqn:1system} satisfying these four conditions.  This estimate involves a restriction to ``the small/medium frequency region" that we clarify later.  The second main result is  the construction of high order approximate solutions for a particular system like  \eqref{eqn:1system}, namely \eqref{eqn:msystem} which satisfies these four conditions.  This result is stated in Theorem \ref{apsothm}.   This theorem is interesting because it shows that high order approximate solutions can be constructed even in situations where we don't have an unrestricted energy estimate. However, due to the lack of uniform existence results for the system \eqref{eqn:1system} and \eqref{eqn:msystem}, we cannot justify the geometric optics solutions.

\subsection{Singular System and Estimations}
We here establish the basic setup for studying the problem. Steps between \eqref{i5} and \eqref{iteintro} is carried out by Williams \cite{W4}.
We first rewrite \eqref{eqn:1system} by applying $ B_2^{-1}$ to both sides of the first equation, for a different $F$ we have:
\begin{align}\label{i5}
\begin{split}
&D_{x_2} u+A_0 D_tu+A_1D_{x_1}u-iB_2^{-1}\mathcal{D}\left(\frac{\phi_{in}}{\eps}\right)u=F(t,x,\frac{\phi_0}{\eps})\\
&Bu= G(t,x_1,\frac{\phi_0}{\eps})\text{ on }x_2=0\\
&u=0\text{ in }t<0,
\end{split}
\end{align}

where $D_{x_i} = \frac{1}{i}\D_{x_i}$, $A_0 = B_2^{-1}$, $A_1 = B_2^{-1}B_1$.  We look for solutions of the form $u(t,x)=U(t,x,\frac{\phi_0}{\eps})$ for \eqref{i5}, where $U(t,x,\theta)$ is periodic in $\theta$.  This benefits us by providing bounded sobolev norms of the right hand side.  Plugging $U(t,x,\theta)$ into \eqref{i5} yields the corresponding  \textit{singular system}:
\begin{align}\label{i6}
\begin{split}
&D_{x_2}U+A_0(D_t+\frac{\sigma_l}{\eps}D_\theta)U+A_1(D_{x_1}+\frac{\eta_l}{\eps}D_\theta)U-iB_2^{-1}\cD\left(\frac{\omega_N(\beta_l)}{\eps}x_2+\theta\right)U=F(t,x,\theta)\\
&BU=G(t,x_1,\theta)\text{ on }x_2=0\\
&U=0\text{ in }t<0.
\end{split}
\end{align}
We observe that a solution of the singular system immediately yields a solution to the original system. 

It is shown in \cite{W4} that we can reduce the first equation of \eqref{i6} to:
\begin{align} \label{1.7}
D_{x_2}U+A_0(D_t+\frac{\sigma_l}{\eps}D_\theta)U+A_1(D_{x_1}+\frac{\eta_l}{\eps}D_\theta)U-i\left(\sum_{r\in\ZZ\setminus 0} \alpha_r e^{i\left(r\frac{\omega_N(\beta_l)}{\eps}x_2+r\theta\right)}\right)B_2^{-1}MU=F(t,x,\theta).
\end{align}
where $M$ is any constant $N\times N$ matrix.

We want to study the Laplace-Fourier transform in $(t,x_1, \theta)$ of the singular system \eqref{i6}, with reduction \eqref{1.7}.  
We can write 
\begin{align} 
U(t,x,\theta)&=\sum_{k\in\ZZ}{U_k}(t,x)e^{ik\theta},\quad
F(t,x,\theta)=\sum_{k\in\ZZ}{F_k}(t,x)e^{ik\theta},\quad
G(t,x_1,\theta)=\sum_{k\in\ZZ}{G_k}(t,x_1)e^{ik\theta}.
\end{align} 
Define $\zeta$, the dual variable of $(t,x_1)$:
\begin{align}\label{i8}
\zeta:=(\tau,\eta):=(\sigma-i\gamma,\eta), \text{ where }(\sigma,\eta)\in\RR^2,\; \gamma\geq 0.
\end{align}
Let $V_k(x_2,\zeta):=\widehat{U_k}(\zeta,x_2)$, the Laplace-Fourier transform of $U_k$.  We define:
\begin{align} \label{s3b}
X_k:=\zeta+\frac{k\beta_l}{\eps}\text{ and }\cA(\zeta)=-(A_0\tau+A_1\eta),
\end{align}
Using above definitions, equations \eqref{i6} and \eqref{1.7}, we have, for each $k$, a corersponding singular system for $V_k$:
\begin{align}\label{i9}
\begin{split}
&D_{x_2}V_k-\mathcal{A}(X_k)V_k=i\sum_{r\in\ZZ\setminus 0}  \alpha_r e^{ir\frac{\omega_N(\beta_l)}{\eps}x_2}B_2^{-1}MV_{k-r}+\widehat{F_k}(x_2,\zeta)\\
&BV_k=\widehat{G_k}(\zeta)\text{ on }x_2=0.
\end{split}
\end{align}
With the help of proposition \ref{finiter},  we obtain an iteration estimate \footnote{The name \textit{iteration estimate} goes back to \cite{W4}, but we don't iterate this estimate here.} (proposition \ref{tt30}) of the following form.  For  $\gamma \geq \gamma_0 > 0$, we have:  
\begin{equation} \label{iteintro}
||\chi V_k||\leq \frac{C}{\gamma}\sum_{r\in\ZZ\setminus 0}\sum_{t\in \ZZ}||\alpha_r\alpha_t\DD(\eps,k,k-r)\chi V_{k-r-t}||+\frac{C}{\gamma^2}\left|\chi\widehat{F_k}|X_k|\right|_{L^2}+\frac{C}{\gamma^{3/2}}\left|\chi \widehat{G_k}|X_k|\right|_{L^2}
\end{equation}
where $\chi$ is the characteristic function for $|\zeta|$ small compared to $\eps^{-1}$,  namely the region $|\zeta| \leq C\eps^{\alpha -1}$ for a fixed $0< \alpha <1$, $||\cdot ||$ is a modified $L^2(x_2,\sigma,\eta)$ norm,  the constants $C$ and $\gamma_0$ is independent of $(\eps, \zeta, k)$, and the $\alpha_r$ are as in \eqref{i9}.  The global amplification factor, $\DD$,  is defined in \ref{t29}.  Estimate \eqref{iteintro} is an important step toward proving the uniform estimate of $U(t,x,\theta)$ in \eqref{i6} (theorem \ref{mainthm}):\\
\begin{align} \label{a}
|\chi_D U^\gamma|_{L^2(t,x,\theta)} + \left|\frac{\chi_DU^\gamma (0)}{\sqrt{\gamma}}\right|_{L^2(t,x_1,\theta)} \leq 
K\left[ \frac{1}{\gamma^2}(\sum_{k\in\ZZ}\left| \chi |X_k|\widehat{F_k}\right|^2_{L^2(x_2,\sigma,\eta)})^{1/2} + \frac{1}{\gamma^{3/2}}(\sum_{k\in\ZZ}\left| \chi|X_k|\widehat{G_k}\right|^2_{L^2(\sigma,\eta)})^{1/2}\right]
\end{align}
where $\chi_D$ is the Fourier multiplier corresponding to $\chi$.

\begin{rem}
\begin{enumerate}
\item Techniques of \cite{CGW14} can be used to prove an estimate of the form \eqref{a} where $\gamma_0=\gamma_0(\eps)\to \infty$ as $\eps\to 0$, but such estimates can't be used to justify WKB expansions.  The main result of this paper is a step toward obtaining an estimate with $\gamma_0$ independent of $\eps$.   The analysis also indicates some obstacles to obtaining estimates with no restriction on $\zeta$. 
\item Since we don't have uniform existence results for the system \eqref{i6},   estimates \eqref{iteintro}, \eqref{a} are stated as a priori estimates. 
\item It is interesting that it is still possible to construct arbitrarily high order WKB solutions in this situation.  However, to justify those solutions (showing the WKB solutions are close to the exact solutions),  we need the estimate \eqref{a} without the $\chi$ and the existence of exact solutions. 
\item We can apply estimate \eqref{a} to get an ``restricted" existence result for a particular singular system that is different from \eqref{i6} by a classical duality argument based on energy estimate  for the corresponding ``backward" or adjoint boundary value problem together with the Riesz representation theorem.  This is discussed informally in section \ref{ext}.
\end{enumerate}
\end{rem}

\begin{notn}
Consider functions $f(x), g(x)$ where $x \in D$ for some domain $D$.  
\begin{align}
f \sim 1 &\iff C_1 \leq |f(x)| \leq C_2 \text{ on } D\\
f \lesssim g &\iff |f(x)| \leq C|g(x)| \text{ on } D
\end{align}
for constants $C, C_1,\ C_2$ independent of $x$.
\end{notn}

\subsection{Main steps and difficulties}\hfill \\

Here we briefly  discuss the obstacles   created by the failure of the uniform Lopatinski condition and the presence of two incoming phases.  To analyze the singular system \eqref{i9}, we diagonalize the equation in the neighbourhood of $\beta_l$ where ULC holds and write down explicitly the integral formula of solutions (\eqref{a5}, \eqref{a6}), which is of the form:
\begin{align}\label{a5intro}
w^+_k(x_2,\zeta)=\sum_{r\in\ZZ\setminus 0}\int^\infty_{x_2}e^{i\xi_+(\eps,k)(x_2-s)+ir\frac{\omega_N(\beta_l)}{\eps}s}\ar[a(\eps,k,k-r)w^+_{k-r}(s,\zeta)+b(\eps,k,k-r)w^-_{k-r}(s,\zeta)]ds,
\end{align}
\begin{align}\label{a6intro}
\begin{split}
&w^-_k(x_2,\zeta)=-\sum_{r\in\ZZ\setminus 0}\int^{x_2}_0e^{i\xi_-(\eps,k)(x_2-s)+ir\frac{\omega_N(\beta_l)}{\eps}s}\ar[c(\eps,k,k-r)w^+_{k-r}(s,\zeta)+d(\eps,k,k-r)w^-_{k-r}(s,\zeta)]ds-\\
&e^{i\xi_-(\eps,k)x_2}[Br_-(\eps,k)]^{-1}Br_+(\eps,k)\sum_{r\in\ZZ\setminus 0}\int^\infty_{0}e^{i\xi_+(\eps,k)(-s)+ir\frac{\omega_N(\beta_l)}{\eps}s}\ar[a(\eps,k,k-r)w^+_{k-r}(s,\zeta)+\\
&\qquad\qquad\qquad b(\eps,k,k-r)w^-_{k-r}(s,\zeta)]ds+e^{i\xi_-(\eps,k)x_2}[Br_-(\eps,k)]^{-1}\hat G_k(\zeta).
\end{split}
\end{align}
where $V_k = (w_k^+\ w_k^-)$, $a,\ b,\ c,\ d$ are matrices that $\sim 1$; $\xi_+(\zeta) = \text{diag}\{{\omega_1(\zeta),...,\omega_{N-p}}(\zeta)\}$, $\xi_-(\zeta) =  \text{diag}\{{\omega_{N-p+1}(\zeta),...,\omega_{N}}(\zeta)\}$, where $\omega_j$ are the eigenvalues of the matrix $\mathcal{A}(\zeta)$; $r_+ = (r_1,...r_{N-p}),\ r_- = (r_{N-p+1},...,r_N)$ where $r_j$ are the normalized eigenvector of the matrix $\mathcal{A}(\zeta)$. The indexes $\mathcal{O}= \{1,...,N-p\}$, $\mathcal{I}= \{N-p+1,...,N\}$ corresponds to outgoing, incoming phases respectively.

The failure of uniform Lopatinski condition as specified in assumption \ref{lopassump} implies that the Lopatinski determinant $\Delta(\beta_l) =  0$ (definition \ref{lopdeterm}).  For functions $f$ depending on $\zeta$, we use notation $f(\eps,k)$ to denote $f(X_k)$ to highlight the dependence on $\eps, k$ (notation \ref{notation}).  Lemma \ref{t6} shows that:
\begin{equation}
|Br_-(\eps,k)|^{-1} \lesssim |\Delta(\eps,k)|^{-1} \sim \frac{|X_k|}{|\tau - c_+(\beta_l)\eta|}
\end{equation}
which is of size $\geq \frac{1}{\eps}$ for $\zeta$ near $\beta_l$. This yields an obvious obstacle when estimating the integral equations.  To estimate this term, \cite{W4} proposed the usage of $E_{i,j}(\eps,k,k-r)$ (definition \ref{defneij}):
\begin{equation}
E_{i,j}(\eps,k,k-r;\beta_l):= \omega_i(\eps,k;\beta_l) - \frac{r\omega_N(\beta_l)}{\eps}-\omega_j(\eps,k-r;\beta_l), \text{ where } i\in \mathcal{O}, j\in \mathcal{I}
\end{equation}
Notice that we can rewrite a component of the last line of the integral equations  \eqref{a6intro} as  $e^{-iE_{i,j}(\eps,k,k-r)s}$,  provided we replace $w_{k-r}^-$ by $e^{-i\xi_-(\eps,k-r)x_2}w_{k-r}^-$.  An integration by part argument using 
\begin{equation}
e^{-iE_{i,j}s} = -\frac{1}{i}E_{i,j}^{-1} \frac{d}{ds}e^{-iE_{i,j}s}
\end{equation}
gives rise to the term $E_{i,j}^{-1}$. This core step allows us to control $|\Delta(\eps,k)|^{-1}$ with $|E_{i,j}^{-1}|$ on regions where $E_{i,j}(\eps,k,k-r)$ is large. We'll see that the existence of more than one incoming phases produces more difficulties when establishing a useful lower bound on some regions.  Actually we were able to prove such lower bounds only when $|\zeta| \lesssim \eps^{\alpha-1}$ for some  $0 < \alpha <1$ (proposition \ref{finiter}).

The failure of  the uniform Lopatinski condition also gives rise to more difficulties  in constructing geometric optics solutions. Since the boundary matrix $B$ in \eqref{i5} fails to be an isomorphism between $\EE^s(\zeta)$ and $\CC^p$ where $\EE^s(\zeta)$ is the stable subspace of $i\mathcal{A}(\zeta)$(definition \ref{stablesubspace}),  we need further analysis on the boundary equations using the approach of \cite{W4},  which is presented in section \eqref{abae}.

\newpage
\section{Definitions and Assumptions} \label{sec2}

In this section we establish relevant definitions, assumptions, and results quoted directly from \cite{W4}.

\begin{assump} \label{stricthyperbolicity} \cite{W4} (Strict Hyperbolicity) The $B_j$ are real matrices, and there exist real valued functions $\lambda_j(\eta,\xi), j = 1,...,N$ that are analytic on $\RR^2 \setminus 0$  and homogeneous of degree one such that
\begin{equation}
\det(\sigma I + B_1 \eta + B_2 \xi) = \prod_{k=1}^N(\sigma + \lambda_k(\eta,\xi)) \text{ for all } (\eta,\xi)\in \RR^2 \setminus 0
\end{equation}
Moreover, we have 
\begin{equation}
\lambda_1(\eta,\xi)<\lambda_2 (\eta,\xi) <...<\lambda_N(\eta,\xi) \text{ for all } (\eta,\xi)\in \RR^2\setminus 0
\end{equation}
\end{assump}

\begin{assump}\cite{W4}\label{BNp}
The matrix $B_2$ is invertible and has $p$ positive eigenvalues, where $1\leq p\leq N-1$.  The boundary matrix $B$ is $p\times N$, real,  and of rank $p$. 
\end{assump}

We define the sets of frequencies:
\begin{equation}
	\begin{aligned}
&\Xi := \{\zeta := (\sigma - i\gamma, \eta)\in\CC \times \RR \setminus (0,0): \gamma\geq 0\} \\
&\Xi_0 := \{(\sigma, \eta)\in\RR \times \RR \setminus (0,0)\} = \Xi \cap \{\gamma = 0\} 
\end{aligned}
\end{equation}

\begin{defn}\label{stablesubspace}
Let $K$ be an $n \times n$ matrix with entries in $\CC$.  If $\lambda$ is an eigenvalue of $K$ of algebraic multiplicity $m \in \NN$, the \textbf{generalized eigenspace} associated to $\lambda$ is $G_\lambda := \{x \in \CC^n: (K- \lambda I)^m x = 0\}$. The \textbf{stable subspace} of $K$ is the direct sum of the generalized eigenspaces associated to eigenvalues of $K$ with real part $< 0$.
\end{defn}

\begin{defn} \label{defnhyperbolic} \cite{W4}
\begin{enumerate}
\item The \textbf{hyperbolic region $\mathcal{H}$} is the set of all $(\sigma, \eta) \in \Xi_0$ such that the matrix $\mathcal{A}(\sigma, \eta)$ is diagonalizable with real eigenvalues.
\item Let \textbf{$G$} denote the set of all $(\sigma, \eta, \xi) \in \RR \times \RR^2$ such that $(\eta,\xi) \neq 0$ and there exists an integer $k \in \{1,...,N\}$ satisfying:
\begin{equation}
\sigma + \lambda_k(\eta,\xi) = \frac{\D\lambda_k}{\D \xi}(\eta,\xi) = 0.
\end{equation}
If $\pi(G)$ denotes the projection of $G$ on the first two coordinates (that is $\pi(\sigma,\eta,\xi) = (\sigma,\eta)$for all $ (\sigma,\eta,\xi)$), the glancing set $\mathcal{G}$ is $\mathcal{G} := \pi(G) \subset \Xi_0$.
\end{enumerate}
\end{defn}

\begin{prop} $[Kre70]$ Let the assumption \ref{stricthyperbolicity} and \ref{BNp} be satisfied. Then for all $\zeta \in \Xi \setminus \Xi_0$, the matrix $i\mathcal{A}(\zeta)$ has no purely imaginary eigenvalue and its stable subspace $\EE^s(\zeta)$ has dimension $p$. Furthermore, $\EE^s$ defines an analytic vector bundle over $\Xi \setminus \Xi_0$ that can be extended as a continuous vector bundle over $\Xi$.  
\end{prop}

We define $\EE^s(\sigma,\eta)$ to be the continuous extension of $\EE^s(\zeta)$. It is shown in \cite{Met} that 
\begin{enumerate}
\item The hyperbolic region $\mathcal{H}$ doesn't contain any glancing point.
\item Away from the glancing point, $\EE^s(\zeta)$ depends analytically on $\zeta$.
\end{enumerate}

\begin{defn} \label{ulcdefn} \cite{Kre70}
Let $p$ be the number of positive eigenvalues of $B_2$, and let 
\begin{equation}
L(\D) = \D_t + B_1\D_{x_1} + B_2\D_{x_2}
\end{equation}
The problem $(L(\D),B)$ is said to be uniformly stable or to satisfy the uniform Lopatinski condition (ULC) if 
\begin{equation} \label{eqn:defnlop}
B: \EE^s(\zeta) \rightarrow \CC^p
\end{equation}
is an isomorphism for all $\zeta \in \Sigma$. Similarly, we say $(L(\D),B)$ satisfies the ULC on $\Xi$ (respectively, on a closed conic subset $\Gamma \subset \Xi$) if \eqref{eqn:defnlop} is an isomorphism on $\Sigma$ (respectively, on the subset of $\Sigma$ corresponding to $\Gamma$). 
\end{defn}

For a fixed $\beta = (\underline{\sigma},\underline{\eta}) \in \mathcal{H}$, we have the following result which follows from assumption \ref{stricthyperbolicity}.
\begin{prop} \cite{W4}
There exists $\delta>0$ such that on the conic neighborhood of $\beta$:
\begin{equation} \label{conicnbhd}
\Gamma^+_{\delta}(\beta) = \left\lbrace \zeta \in \Xi: \bigl| \frac{\zeta}{|\zeta|}-\frac{\beta}{|\beta|}\bigr| \leq \delta \right\rbrace
\end{equation}
the matrix $\mathcal{A}(\zeta) = -(A_0\tau + A_1 \eta)$ has $N$ distinct eigenvalues $\omega_j(\zeta)$ and corresponding eigenvectors $R_j(\zeta)$ such that:
\begin{equation}  
\mathcal{A}_j(\zeta)R_j(\zeta) = \omega_j(\zeta)R_j(\zeta),\ \ j=1,...,N \text{ on } \Gamma^+_{\delta}(\beta)
\end{equation}
Functions $\omega_j$, $R_j$ map are homogeneous of degree one,  analytic in $\tau$, and $C^\infty$ in $\eta$.  The corresponding normalized vectors
\begin{equation}
r_j(\zeta):= \frac{R_j(\zeta)}{|R_j(\zeta)|}
\end{equation}
are $C^\infty$ in $(\sigma,\gamma,\eta)$.  
\end{prop}

\begin{defn} \label{defnincoming} \cite{W4}
For each root $\omega_j(\beta) = \underline{\omega}_j$ there corresponds a unique integer $k_j \in \{1,...,N\}$ such that $\underline{\sigma} + \lambda_{k_j}(\underline{\eta}, \underline{\omega}_j) = 0$. We define the following real phases $\phi_j$ and their associated group velocities $v_j$:
\begin{equation}
\phi_j(x):= \phi_0(t,y) + \underline{\omega}_j x_2, \quad v_j := \nabla \lambda_{k_j}(\underline{\eta}, \underline{\omega}_j), \quad \forall j = 1,2,...,N
\end{equation}
The phase $\phi_j$ is incoming if $\D_\xi \lambda_{k_j}(\underline{\eta},\underline{\omega}_j) > 0$ , and is outgoing if $\D_\xi \lambda_{k_j}(\underline{\eta},\underline{\omega}_j) < 0$. If the phase $\phi_j$ is incoming (resp. outgoing), we refer to the corresponding frequency $\underline{\omega}_j$ as incoming (resp. outgoing).  For $\beta \in \mathcal{H}$, there are exactly $N-p$ outgoing phases and $p$ incoming phases.  We denote the corresponding set of indices by $\mathcal{I} = \{N-p+1,...,N\}$ for incoming,  and $\mathcal{O} = \{1,2,...,N-p\}$ for outgoing.
\end{defn}

\begin{prop} \cite{W4}
The $\omega_j$ are real-valued for $\gamma=0$ and can be  divided into two groups according as $j\in \mathcal O$ or $\mathcal I$.     There exists  a constant $c>0$ such that for $\zeta\in\Gamma^+_\delta(\beta)$
 \begin{align}\label{s5}
 \begin{split}
 &\mathrm{Im}\;\omega_j(\zeta)\leq -c\gamma \text{ for }j\in\mathcal{O}\\
 &\mathrm{Im}\;\omega_j(\zeta)\geq c\gamma \text{ for }j\in \mathcal{I}.
\end{split}
\end{align}
\end{prop}

We define the $N\times (N-p)$ matrix $r_+(\zeta)$ of normalized eigenvectors corresponding to the outgoing phases and the $N\times p$ matrix $r_-(\zeta)$  corresponding to the incoming phases by 
\begin{align}\label{s6}
r_+=\begin{pmatrix}r_1&r_2&\dots&r_{N-p}\end{pmatrix}, \qquad r_-=\begin{pmatrix}r_{N-p+1}&\dots&r_{N}\end{pmatrix} \text{ on }\Gamma^+_\delta(\beta).
\end{align}
We similarly define $R_\pm(\zeta)$ using the  eigenvectors $R_j(\zeta)$. 
We also define the $N \times N$ matrix
\begin{equation} \label{szeta}
S(\zeta) = (r_+(\zeta)\quad r_-(\zeta)) = (r_1, ... ,r_N) \text{ on } \Gamma_\delta^+(\beta)
\end{equation}
\begin{prop}\cite{W4}  \label{esrj}
For $\zeta \in \Gamma_\delta^+(\beta_l)$, the stable subspace $\EE^s(\beta_l)$ admits decomposition:
\begin{equation}
\EE^s(\beta_l) = \bigoplus_{j \in \mathcal{I}} r_j(\zeta)
\end{equation}
and $r_j$ can be (and are) taken to be real vectors.
\end{prop}

\begin{defn} \label{defnres}\cite{W4} For $i \in \mathcal{O}, j,N \in \mathcal{I}$ with $j \neq N$, we say that the phases $(\phi_j,\phi_N,\phi_i)$ exhibit a resonance if there exist $p,q \in \ZZ\setminus 0$ such that 
\begin{equation}
p\phi_j + q \phi_N = (p+q) \phi_i
\end{equation}
\end{defn}
Let 
\begin{equation} \label{Omegadefn}
\Omega_{i,j} := \frac{\omega_i(\beta_l)-\omega_N(\beta_l)}{\omega_j(\beta_l)-\omega_i(\beta_l)} \quad i \in \mathcal{O}, j\in \mathcal{I} \setminus \{N\}
\end{equation}
\begin{rem}
If there exists $p,q \in \ZZ\setminus 0$ such that $(\phi_j,\phi_N,\phi_i)$ exhibit a resonance, then 
\begin{equation}
p\phi_j + q \phi_N = (p+q) \phi_i \iff 
p\omega_j(\beta_l) + q \omega_n(\beta_l) = (p+q) \omega_i(\beta_l) \iff
\frac{p}{q} =  \frac{\omega_i(\beta_l)-\omega_N(\beta_l)}{\omega_j(\beta_l)-\omega_i(\beta_l)} = \Omega_{i,j}
\end{equation}
Above equations show that there is resonance $\iff$ $\Omega_{i,j} \in \QQ$.
\end{rem}
In this paper, we only study the case when there is no resonance. Next we define Lopatinski determinant.

\begin{defn}\label{lopdeterm} \cite{W3} 
For $\zeta \in \Gamma^+_\delta(\beta)$, define the analytic Lopatinski determinant
\begin{equation}
\Delta_a(\zeta) = \det BR_{-}(\zeta) 
\end{equation}
and the normalized Lopatinski determinant:
\begin{equation}
\Delta(\zeta) = \det Br_{-}(\zeta) 
\end{equation}
\end{defn}

We next make weak stability assumption on the problem $(L(\D),B)$ when ULC fails.  This assumption is not in effect in section 3 where we study the case when ULC holds. 

\begin{assump} \cite{W4} \label{lopassump}
\begin{enumerate}
\item For all $\zeta \in \Xi \backslash \Xi_0, \ker B\cap \EE^s(\zeta) = \{0\}$.
\item The set $\Upsilon_0 := \{\zeta \in \Sigma_0: \ker B \cap \EE^s(\zeta) \neq \{0\}\} = \{\beta_l , -\beta_l\}$ is included in the hyperbolic region $\CH$.
\item  There exists a neighborhood $\Gamma^+_\delta(\beta_l)$ as in  \eqref{conicnbhd} on which functions $\omega_j$, $R_j$, $\Delta$ are defined and we have:
\begin{equation}
\Delta(\beta_l) = 0 \text{ and } \D_\tau\Delta(\beta_l)\neq 0 
\end{equation}
\end{enumerate}
\end{assump}

\begin{notn}
Throughout this paper, the notation $L^2(t,x,\theta)$ denotes the space $L^2(\Omega \times \TT_\theta)$ where $\TT_\theta = \RR^3 /2 \pi \ZZ$. We similarly define the space $H^\infty (t,x,\theta)$ etc.
\end{notn}


\newpage
\section{Geometric Optics Solutions for The Hyperbolic System when ULC Holds} \label{section1}
In this section, we consider the case when uniform Lopatinski condition holds. We construct geometric optics solutions for a $3 \times 3$ system with oscillatory term term  $e^{i\frac{\phi_3}{\epsilon}}M$ and use an estimate from \cite{Kre70} to show that this solution is close to the exact solution.  This section is a warm-up to be contrasted with the more involved construction of the final section, which deals with a problem where ULC fails.

\subsection{Introduction} \label{introduction} \hfill \\

We study the following $3 \times 3$ hyperbolic system on $\Omega_T =(-\infty,T) \times \{(x_1,x_2):x_2 \geq 0\}$:

\begin{equation} \label{eqn:system}
	\begin{aligned}
&L(\partial)u + e^{i\frac{\phi_3}{\epsilon}}Mu:= \partial_tu + B_1\partial_{x_1}u + B_2\partial_{x_2}u + e^{i\frac{\phi_3}{\epsilon}}Mu = 0 \text{ in } x_2>0 \\ 
&Bu := \epsilon G \left( t,x_1,\frac{\phi_0}{\epsilon} \right) \text{ on }x_2 = 0 \\ 
&u = 0 \text{ in } t<0 
	\end{aligned}
\end{equation}
where all $B_j$'s are $3 \times 3$ matrices, $B_2$ is invertible, $M$ is a constant matrix, $\phi_1$ is outgoing, $\phi_2,\phi_3$ is incoming, and $G$ is periodic in $\theta_0 = \frac{\phi_0}{\eps}$,  where:\\
\begin{equation}
	\begin{aligned}
	&\phi_m (t,x_1,x_2) = \beta_l \cdot (t, x_1) + \omega_m(\beta_l)x_2 \\
	&d\phi_m = (\beta_l, \omega_m(\beta_l))
	\end{aligned}
\end{equation}
and $\beta_l$ being a fixed point in $\mathcal{H}$. We write $U\left(t,x,\frac{\Phi}{\epsilon} \right) = U\left(t,x,\theta \right)|_{\theta = \frac{\Phi}{\epsilon}}$ where $\Phi = (\phi_1, \phi_2, \phi_3)$ and $\theta = (\theta_1,\theta_2,\theta_3)$, $\theta_0 = \frac{\phi_0}{\epsilon}$. We set: 
\begin{equation}
	\begin{aligned}
&L(\partial) =  \partial_t + B_1\partial_{x_1} + B_2\partial_{x_2} \quad
L(\sigma,\eta,\xi) = \sigma I + B_1 \eta + B_2 \xi \\
&\mathcal{L}(\partial_\theta) = \sum_{m=1}^3 L(d\phi_m)\partial_{\theta_{m}} 
	\end{aligned}
\end{equation}
Let $\mathcal{A} (\beta_l)$ be the matrix:
\begin{equation}
	\begin{aligned}
	\mathcal{A} (\beta_l) = -(A_0 \sigma_l + A_1 \eta_l) \text{, where } A_0 = B_2^{-1}, A_1 = B_2^{-1}B_1
	\end{aligned}
\end{equation}
By strict hyperbolicity, the matrix $\mathcal{A}(\beta_l)$ is diagonalizable with eigenvalues $\omega_m(\beta_l) = \omega_m$, $m = 1,...,3$, and we observe that the eigenspace of $\mathcal{A}(\beta_l)$ for $\omega_m$ is precisely the kernel of $L(d\phi_m)$. 


\subsection{Tools for constructing approximate solutions} \cite{W4}\label{tfcas}

This section establishes lemmas, definitions, and assumptions for constructing approximate solutions.

\begin{lem}\label{E^s} $[CG10]$ The (extended) stable subspace $\EE^s(\beta_l)$ admits the decomposition
\begin{equation}
\EE^s(\beta_l) = \oplus_{m\in \mathcal{I}} \Ker L(d\phi_m)
\end{equation}
and each vector space in this decomposition is of real type (it admits a basis of real vectors).
\end{lem}
\begin{lem}\label{decomplem} $[CG10]$ The following decomposition hold
\begin{equation} \label{eqn:decomp}
\CC^3 = \oplus_{m=1}^3 \Ker L(d\phi_m) = \oplus_{m=1}^3 B_2 \Ker L(d\phi_m)
\end{equation}
and each vector space in this decomposition is of real type. \\
We let $P_m$, respectively $Q_m$ ($m = 1,2,3$), denote the projectors associated with the first, respectively, second decomposition in \ref{eqn:decomp}. For each $m$, we have Im$L(d{\phi_m})=$Ker$Q_m$.
\end{lem}

\begin{prop} \cite{W4}
We can define operator $R_m$ uniquely by:
\begin{align}
R_mL(d\phi_m)= I - P_m,\quad
L(d\phi_m)R_m = I - Q_m,\quad
P_mR_m = 0,\quad
R_mQ_m = 0.
\end{align}
If we choose for each $m$ a real vector $r_m$ that spans $\text{Ker}L(d\phi_m)$, and  choose a real row vector $\ell_m$ that satisfy $\ell_mL(d\phi_m) = 0$ with the normalization $\ell_mB_2r_{m'} = \delta_{mm'}$, then $P_m$, $Q_m$, $R_m$ satisfies:
\begin{align}
\forall X\in \CC^3,\quad
R_mX = \sum_{m\neq m'}\frac{\ell_{m'}X}{\omega_m-\omega_{m'}}r_{m'},\quad
P_mX = (\ell_mB_2X)r_m,\quad
Q_mX = (\ell_mX)B_2r_m.
\end{align}
\end{prop}
We define projection operators $E_Q$, $E_P$ on $H^{\infty} := H^{\infty}(\Omega_T \times \TT_\theta)$ and partial inverse operators $R$ of $\mathcal{L}(\D_\theta)$ on a proper subspace. We define 
\begin{equation}
E_P = E_{P_0} + \sum_{m=1}^3 E_{P_{m}},\quad E_{P_{in}} = E_{P_2}+E_{P_3},\quad E_{P_{out}} = E_{P_1}
\end{equation}
where $E_{P_0}$ picks out the mean and $E_{P_m}$ picks out the pure $\theta_m$ modes and then projects with $P_m$. For example, if we write:
\begin{equation} \label{eqn:Unotn}
U(t,x,\theta) = \underline{U}(t,x) + U^*(t,x,\theta) = \underline{U}(t,x) +\sum_{m=1}^3 U^m(t,x,\theta_m) + U^{nc}(t,x,\theta_1,\theta_2,\theta_3)
\end{equation}
where $\underline{U}$ is the mean of $U$, each $U^m$ has pure $\theta_m$ oscillations with mean zero, and $U^{nc}$ consists of only noncharacteristic modes in the Fourier series of $U$, we have:
\begin{align}
&E_{P_0}U = \underline{U},\quad E_{P_m}U = P_mU^m(t,x,\theta_m),\quad m=1,2,3\\
&(I-E_P)U = \sum_{m=1}^3(I-P_m)U^m + U^{nc}
\end{align}
We define $E_Q$ in the same way but we take the projection with $Q_m$ instead. 
Now to define the operator $R$,  we define the following subspace of $H^{\infty}$:
\begin{defn}
Define $\mathcal{N} = \{(k,l) \in \ZZ\setminus 0 \times \ZZ\setminus 0 \}$. Let $\mathcal{H}^{\infty}$ be the subspace of $H^{\infty}(t,x,\theta)$ given by\footnote{Here $\theta$ is a place holder for $\frac{\Phi}{\eps}$.}:
\begin{equation}
\mathcal{H}^{\infty} := \{U\in H^{\infty}: U^{nc} = \sum_{(k,l)\in \mathcal{N}} c^{k,l}(t,x)e^{i(k\theta_2 + l\theta_3)}\}
\end{equation}
\end{defn}


\begin{rem}
\begin{enumerate}
\item We will see later the profiles $U(t,x,\theta)$ that we construct lie in $\mathcal{H}^\infty$.
\item In this section, the profiles actually lie in a subset of $\mathcal{H}^\infty,$:
\begin{equation}
\{U\in H^{\infty}: U^{nc} = \sum_{(k,l)\in \mathcal{N}, l \geq 1} c^{k,l}(t,x)e^{i(k\theta_2 + l\theta_3)}\}
\end{equation}
since the oscillatory term in \eqref{eqn:system} only has positive Fourier spectrum.
\end{enumerate}
\end{rem}

We make the following small divisor assumption for both \eqref{eqn:system} and \eqref{eqn:msystem}.

\begin{assump} (Small divisor assumption)\label{sdassump}\footnote{This assumption is necessary for \ref{eqn: RUnc} to be well defined.}
There exists a constant $C>0$ and $a \in \RR$  such that $\forall (k,l) \in \mathcal{N}$ the following holds:
\begin{equation}
||L(ikd\phi_2+ild\phi_3)^{-1}|| \geq C|(k,l)|^{a}
\end{equation} 
\end{assump}

For $U \in \mathcal{H}^{\infty}$ we define $R$ by:
\begin{equation}
	\begin{aligned}
&R(\underline{U}) = 0\\
&R(U^m) = \D_{\theta_m}^{-1}R_mU^m\\
&R(U^{nc}) = \mathcal{L}(\D_\theta)^{-1}U^{nc}
\end{aligned}
\end{equation}
where $\D_{\theta_m}^{-1}$ denotes operator of finding the unique mean zero antiderivative in $\theta_m$, and a calculation shows that
\begin{equation} \label{eqn: RUnc}
\mathcal{L}(\D_\theta)^{-1}U^{nc} = \sum_{(k,l)\in \mathcal{N}} L^{-1}(ikd\phi_2 + ild\phi_3)c_{k,l}(t,x)e^{i(k\theta_2 + l\theta3)}
\end{equation}
The small divisor assumption \ref{sdassump}, then shows $R: \mathcal{H}^{\infty} \rightarrow \mathcal{H}^{\infty}$. The next proposition follows directly from definition.\\
\begin{prop} \cite{W4}
As operators on $\mathcal{H}^{\infty}$, the operators $\mathcal{L}(\D_\theta), E_P, E_Q,$ and $R$ satisfy:
\begin{equation} \label{eqn:ROp}
	\begin{aligned}
	&E_Q\mathcal{L}(\D_\theta) = \mathcal{L}(\D_\theta)E_P = 0\\
	&R\mathcal{L}(\D_\theta) = I - E_P,\quad \mathcal{L}(\D_\theta)R= I - E_Q\\
	&E_PR = RE_Q = 0
	\end{aligned}
\end{equation}
\end{prop}

\subsection{Profile equations} \hfill \\

We look for approximate solutions to the system \eqref{eqn:system}  of the form:
\begin{equation} \label{eqn:profile}
	\begin{aligned}
	u^\epsilon_a(t,x) = \sum_{k={1}}^J \epsilon^kU_k(t,x,\frac{\Phi}{\epsilon})
	\end{aligned}
\end{equation}
The next proposition summaries the work of this section.
\begin{prop}
Consider the system \eqref{eqn:system} where we assume $Mr_3 = 0$.\footnote{The assumption $Mr_3=0$ is used to decouple the problem, for example in \eqref{eqn:U1mean}, \eqref{eqn:sigma_1^3}.} If assumption \ref{sdassump} is satisfied and  there is no resonance, there exists an approximate solution  $u^\epsilon_a(t,x) = \sum_{k={1}}^J \epsilon^kU_k(t,x,\frac{\Phi}{\epsilon})$ where $U_k \in \mathcal{H}^\infty(t,x,\theta)$, periodic in $\theta$, and for any $n \in \NN$,  by taking $J$ large enough, one can arrange that $u_a^\eps$ satisfies:
\begin{equation}
|u^\eps - u_a^\eps|_{L^\infty(\Omega_T)} \leq C\eps^n
\end{equation}
where $u^\eps$ is the exact solution to \eqref{eqn:system}.
\end{prop}
As proof of the proposition, we first construct $u_a^\eps$.  We observe that the first term in the sum of $u_a^\eps$ satisfies:
\begin{equation}
	\begin{aligned}
L(\partial)\left(\epsilon U_1\left(t,x,\frac{\Phi}{\epsilon}\right)\right)
&=  \partial_t(\epsilon U_1)\left(t,x,\frac{\Phi}{\epsilon}\right) + B_1\D_{x_1}(\epsilon U_1)\left(t,x,\frac{\Phi}{\epsilon}\right)+B_2\D_{x_2}(\epsilon U_1)\left(t,x,\frac{\Phi}{\epsilon}\right)\\  
&+ (\D_t\phi_1+B_1\D_{x_1}\phi_1+B_2\D_{x_2}\phi_1)\D_{\theta_1}U_1\\ 
&+ (\D_t\phi_2+B_1\D_{x_1}\phi_2+B_2\D_{x_2}\phi_2)\D_{\theta_2}U_1\\ 
&+ (\D_t\phi_3+B_1\D_{x_1}\phi_3+B_2\D_{x_2}\phi_3)\D_{\theta_3}U_1\\
	\end{aligned}
\end{equation}
We see that the last three lines is exactly $\CalL(\D_\theta)U_1$. Therefore, plugging \eqref{eqn:profile} into the interior equation in \eqref{eqn:system} gives:
\begin{equation}
	\begin{aligned}
	&\epsilon (L(\D)U_1)(t,x,\frac{\Phi}{\eps}) + (\CalL(\D_\theta)U_1)(t,x,\frac{\Phi}{\eps}) + \epsilon e^{\frac{i\phi_3}{\epsilon}}MU_1(t,x,\frac{\Phi}{\eps})
+ \\ &\eps^2 (L(\D)U_2)(t,x,\frac{\Phi}{\eps}) + \eps(\CalL(\D_\theta)U_2)(t,x,\frac{\Phi}{\eps}) + \eps^2 e^{\frac{i\phi_3}{\epsilon}}MU_2(t,x,\frac{\Phi}{\eps})
+  ... = 0
	\end{aligned}
\end{equation}
The coefficients of same order $\epsilon$ gives the following interior equations of $U_j(t,x,\theta)$
\begin{equation} \label{eqn:interiorprofile}
	\begin{aligned}
	&\CalL(\D_\theta)U_1 = 0\\
	&L(\D)U_1 + \CalL(\D_\theta)U_2 + e^{\frac{i\phi_3}{\epsilon}}MU_1 = 0\\
	&L(\D)U_2 + \CalL(\D_\theta)U_3 + e^{\frac{i\phi_3}{\epsilon}}MU_2 = 0 \\
	&L(\D)U_j + \CalL(\D_\theta)U_{j+1} + e^{\frac{i\phi_3}{\epsilon}}MU_j = 0 \text{ for } j\geq 2
	\end{aligned}
\end{equation}
and for boundary equations on $x_2 = 0,\ \theta_m = \theta_0$, we have:
\begin{equation} \label{eqn:bry}
	\begin{aligned}
	&BU_1 = G\left(t,x_1,\theta_0\right) \\
	&BU_j = 0 \text{ for } j\geq 2.
	\end{aligned}
\end{equation}

Now we apply the operators defined from lemma \ref{decomplem} and the discussion that follows the lemma to above equations.  Applying $E_Q$ to both sides of \ref{eqn:interiorprofile},  we reach the following interior equations: 
\begin{equation}\label{eqn:qequations}
	\begin{aligned}
	&E_Q[L(\D)U_1+e^{i\theta_3}MU_1] = 0\\
	&E_Q[L(\D)U_2+e^{i\theta_3}MU_2] = 0\\
	&E_Q[L(\D)U_j+e^{i\theta_3}MU_j] = 0
	\end{aligned}
\end{equation}

Applying operator $R$ to \ref{eqn:interiorprofile} gives:
\begin{equation} \label{eqn:intequation2}
	\begin{aligned}
	&(I-E_P)U_1 = 0\\
	&(I-E_P)U_2 = -R[L(\D)U_1+e^{i\theta_3}MU_1]\\
	&(I-E_P)U_3 = -R[L(\D)U_2+e^{i\theta_3}MU_2]\\
	\end{aligned}
\end{equation}

As for boundary equations, we first note that the space $\mathbb{E}^s(\beta_l)$ is spanned by $\{r_2,r_3\}$.  Uniform Lopatinski condition holds implies $B: \mathbb{E}^s(\beta_l) \rightarrow \CC^2$ is an isomorphism. Recall that $E_{P_{in}}U_k \in span\{r_2,r_3\} = \mathbb{E}^s(\beta_l)$, thus for $BU_k = H_k$ we have the following boundary equations:
\begin{equation} \label{eqn:bdyeqn}
	\begin{aligned}
	&B\underline{U_k} = \underline{H_k}\\
	&BU_k^*=H_k^*
	\end{aligned}
\end{equation}
 where $H_k$ is read from \eqref{eqn:bry}. We notice that using the definition of projecting operators, the second equation of \eqref{eqn:bdyeqn} can be written as:
\begin{equation} \label{boudaryeqn}
	\begin{aligned} 
	&BE_{P_{in}}U_k = H_k^* - BE_{P_1}U_k - B[(I-E_P)U_k]^*\\
	\end{aligned}
\end{equation}


\subsection{Determining the solutions} \hfill\\

To equations \eqref{eqn:intequation2} with the boundary condition \eqref{boudaryeqn}, we will make use of the following proposition from \cite{W4}:
\begin{prop}\cite{W4} \label{propX}
For $U \in \mathcal{H}^\infty$,  let $E_{P_m} U = \sigma(t,x,\theta_m)r_m$ and let $X_{\phi_m} = \D_{x_2} - \D_{\tau}\omega_m(\beta_l)\D_t - \D_{\eta}\omega_m(\beta_l)\D_{x_t}$ be the transport vector field associated to the phase $\phi_m$. Then 
\begin{equation} \label{eqn:propX}
	\begin{aligned}
	E_{Q_m}L(\D)E_{P_m}U = Q_m(L(\D)\sigma(t,x,\theta_m)r_m) = (X_{\phi_m}\sigma)B_2r_m
	\end{aligned}
\end{equation}
\end{prop}

First, we consider the leading profile. 
\\
Since $(I-E_P)U_1 = 0$, for some scalar profiles $\sigma$ we write:
\begin{equation}
U_1 = \underline{U_1}+\sigma_1^1(t,x,\theta_1)r_1 + \sigma_1^2(t,x,\theta_2)r_2 + \sigma_1^3(t,x,\theta_3)r_3
\end{equation}
\\

1. Determination of $\underline{U_1}$. \\

Taking the mean on both sides of \ref{eqn:qequations}, we have: $E_{Q_0}[L(\D)U_1+e^{i\theta_3}MU_1] = 0$. We note that the mean $L(\D)U_1$ is $L(\D)\underline{U_1}$ and the mean of $e^{i\theta_3}MU_1$ is $\underline{e^{i\theta_3}M\sigma_1^3r_3}$. Thus, using \ref{eqn:bdyeqn} we have:
\begin{equation} \label{eqn:U1mean}
	\begin{aligned} 
	&L(\D)\underline{U_1}+\underline{e^{i\theta_3}M\sigma_1^3r_3} = 0\\
	&B\underline{U_1} = \underline{G}
	\end{aligned}
\end{equation}
We note that this has the coupling involving $\sigma_1^3r_3$.  So the assumption that   $Mr_3 = 0$ 
decouples this problem and will yield a solution for $\underline{U_1}$. 
\\

2. Determining the leading profile $U_1$\\

Using \ref{eqn:qequations}, \ref{eqn:propX}, $U = E_PU$, and the assumption of no resonance we have:
\begin{equation} \label{eqn:interiortransport}
	\begin{aligned} 
	&E_{Q_1}[L(\D)U_1+e^{i\theta_3}MU_1] = X_{\phi_1}\sigma_1^1B_2r_1 = 0\\
	&E_{Q_2}[L(\D)U_1+e^{i\theta_3}MU_1] = X_{\phi_2}\sigma_1^2B_2r_2 = 0\\
	&E_{Q_3}[L(\D)U_1+e^{i\theta_3}MU_1] = X_{\phi_3}\sigma_1^3B_2r_3 + Q_3e^{i\theta_3}M\sigma_1^3r_3 + Q_3e^{i\theta_3}M\underline{U_1}= 0
	\end{aligned}
\end{equation}
From the boundary equation \eqref{eqn:bdyeqn}, we have:
\begin{equation} \label{eqn:boundaryU1}
	\begin{aligned} 
	&B\sigma_1^1r_1 + B\sigma_1^2r_2 + B\sigma_1^3r_3 = G^*
	\end{aligned}
\end{equation}
To solve these equations, we note that \eqref{eqn:interiortransport} are homogeneous transport equations with constant coefficients. Thus, $\sigma_1^1$ and $\sigma_1^2$ are constant  on the integral curve of $X_{\phi_k}$ for $k =1,2$. Since $\phi_1$ is outgoing, $\sigma_1^1 = 0$ at $t < 0$, we have $E_{P_1}U_1 = \sigma_1^1 = 0$. \\
To solve for $\sigma_1^2$, we need a boundary equation. Using $B: \text{span} \{r_2,r_3\} \rightarrow \CC^2$ is an isomorphism,  the boundary equation \ref{eqn:boundaryU1} becomes:
\begin{equation} 
	\begin{aligned} 
	B\sigma_1^2r_2 + B\sigma_1^3r_3 = G^* = c_1(t,x_1,\theta_0)Br_2+c_2(t,x_1,\theta_0)Br_3
	\end{aligned}
\end{equation}
which implies:
\begin{equation} 
	\begin{aligned} 
	&\sigma_1^2\mid_{x_2 = 0,\ \theta_2 = \theta_0} = c_1 \\
	&\sigma_1^3\mid_{x_2 = 0,\ \theta_3 = \theta_0} = c_2
	\end{aligned}
\end{equation}
With those boundary equations, we have uniquely determined $\sigma_1^2$. For the equation of $\sigma_1^3$ in \ref{eqn:interiortransport}, using $Q_mX = l_mXB_2r_m$, we rewrite this equation as: 
\begin{equation}
X_{\phi_1}\sigma_1^3 + e^{i\theta_3}\sigma_1^3\ell_mMr_3 + e^{i\theta_3}\ell_mM\underline{U_1} = 0
\end{equation} 
and we have the following coupling equations for $\sigma_1^3$:
\begin{equation} \label{eqn:sigma_1^3}
	\begin{aligned} 
	&\begin{cases}
	&X_{\phi_1}\sigma_1^3 + e^{i\theta_3}\sigma_1^3\ell_mMr_3 + e^{i\theta_3}\ell_mM\underline{U_1} = 0 \\
	&\sigma_1^3\mid_{x_2 = 0,\ \theta = \theta_3} = c_2\\
	\end{cases}
	\\
	&\begin{cases}
	&L(\D)\underline{U_1}+\underline{e^{i\theta_3}M\sigma_1^3r_3} = 0 \\
	&B\underline{U_1}|_{x_2=0} = \underline{G}
	\end{cases}
	\end{aligned}
\end{equation}
To decouple the problem, we again  use the assumption that $Mr_3 = 0$. Then $\sigma_1^3$ is uniquely determined. \\
\\

3. Determining $U_2$\\

Next, we consider $U_2$. From \eqref{eqn:intequation2}, we have:
\begin{equation} 
	\begin{aligned} 
	U_2 = E_PU_2 + (I - E_PU_2) = \underline{U_2} + \sigma_2^1r_1 + \sigma_2^2r_2 + \sigma_2^3r_3 -R[L(\D)U_{1}+e^{i\theta_3}MU_{1}]
	\end{aligned}
\end{equation}
We observe that $(I-E_P)U_2 = -R(L(\D)U_{1}+e^{i\theta_3}MU_{1})$ is determined by previous terms and is known.  Plugging in this decomposition into the equation $E_Q[L(\D)U_2+e^{i\theta_3}MU_2] = 0$, we have:
\begin{equation} 
	\begin{aligned} 
	E_Q(L(\D)E_PU_2 + e^{i\theta_3}ME_PU_2) &= -E_Q(L(\D)(I-E_P)U_2+e^{i\theta_3}M(I-E_P)U_2) \\
	\end{aligned}
\end{equation}
This yields four equations for $U_2$: 
\begin{equation} \label{eqn:U2underline}
	\begin{aligned} 
	E_{Q_0}(L(\D)E_PU_2 + e^{i\theta_3}ME_PU_2) &= -E_{Q_0}[L(\D)(I-E_P)U_2+e^{i\theta_3}M(I-E_P)U_2] \\
	&= -E_{Q_0}[(I-E_P)U_2] - E_{Q_0}[e^{i\theta_3}M(I-E_P)U_2]
	\end{aligned}
\end{equation}
\begin{equation} \label{eqn:U2terms}
	\begin{aligned} 
	X_{\phi_1}\sigma_2^1B_2r_1 &= -E_{Q_1}[L(\D)(I-E_P)U_2+e^{i\theta_3}M(I-E_P)U_2]\\
	X_{\phi_2}\sigma_2^2B_2r_2 &= -E_{Q_2}[L(\D)(I-E_P)U_2+e^{i\theta_3}M(I-E_P)U_2]\\
	X_{\phi_3}\sigma_2^3B_2r_3 &+ Q_3e^{i\theta_3}M\sigma_2^3r_3 + 
	Q_3e^{i\theta_3}M\underline{U_2} = 
	\\& -E_{Q_3}[L(\D)(I-E_P)U_2+e^{i\theta_3}M(I-E_P)U_2]
	\end{aligned}
\end{equation}

We observe that the first term in \eqref{eqn:U2underline}, $-E_{Q_0}[(I-E_P)U_2]$ picks out the mean of $-[(I-E_P)U_2] = R(L(\D)U_{1}+e^{i\theta_3}MU_{1})$. By the definition of $R$, this term has mean zero.  For the second term, we have \\
$- E_{Q_0}[e^{i\theta_3}M(I-E_P)U_2] = \underline{e^{i\theta_3}MR[L(\D)U_{1}+e^{i\theta_3}MU_{1}]} = \underline{e^{i\theta_3}MR[L(\D)\sigma_1^3r_3+e^{i\theta_3}M(\underline{U_1}+\sigma_1^3r_3)]}$

We note that all equations in \eqref{eqn:U2terms} are evaluated at vector $B_2r_m$, which makes them scalar equations.  We also observe that the right hand sides are known and we can use similar techniques as before. For the first equation in \eqref{eqn:U2terms} involving $\sigma_2^1$, the first term on the right hand side satisfies:
\begin{equation}-E_{Q_1}[L(\D)(I-E_P)U_2] = -E_{Q_1}[L(\D)(-R(L(\D)U_{1}+e^{i\theta_3}MU_{1})] = -E_{Q_1}[L(\D)(-R(L(\D)U_{1})]
\end{equation}
Since $U_1$ has no oscillation in $\theta_1$,  above equation equals to zero. The second term on the right hand side of the equation \eqref{eqn:U2terms} involving $\sigma_2^1$ is $-E_{Q_1}[e^{i\theta_3}M(-R(L(\D)U_{1})] = 0$ since there is no resonance. Thus, the first equation of \eqref{eqn:U2terms} becomes 
\begin{equation}
X_{\phi_1}\sigma_2^1B_2r_1 = 0
\end{equation}
Since $X_{\phi_1}$ is outgoing and $\sigma_2^1 = 0$ when $t < 0$, we have $\sigma_2^1 = 0$. \\
The equation involving $\sigma_2^2$ in \eqref{eqn:U2terms} is an inhomogenenous transport equation with constant coefficients, so it has unique solutions provided a boundary condition.  Boundary condition \eqref{eqn:bry} implies $BE_PU_2 + B(I-E_P)U_2 = 0$.  Using $B\underline{U_2} = 0$ and isomorphism $B: \text{span} \{r_2,r_3\} \rightarrow \CC^2$, we have:
\begin{equation}
B\sigma_2^2r_2 + B\sigma_2^3r_3 = -B(I-E_P)U_2 = c_2^2Br_2 + c_2^3Br_3
\end{equation}
which provides boundary equations $\sigma_2^2 = c_2^2$, $\sigma_2^3 = c_2^3$ and thus gives the unique solution $\sigma_2^2$. 
\\
Similar to the case of $U_1$,  equations involving $\sigma_k^3$ result in the couplings with the following equations involving $\underline{U_k}$:
\begin{equation} 
	\begin{aligned} 
	\begin{cases}
	&X_{\phi_3}\sigma_2^3B_2r_3 + Q_3e^{i\theta_3}M\sigma_2^3r_3 + 
	 Q_3e^{i\theta_3}M\underline{U_2} 
	\\& = -E_{Q_3}[L(\D)(I-E_P)U_2+e^{i\theta_3}M(I-E_P)U_2]\\
	&\sigma_2^3|_{x_2=0, \theta_3 = \theta_0} = c_2^3
	\end{cases}
	\end{aligned}
\end{equation}
\begin{equation} 
	\begin{aligned} 
	&\begin{cases}
	&L(\D)\underline{U_2}+\underline{e^{i\theta_3}M\sigma_2^3r_3} = \underline{e^{i\theta_3}MR[L(\D)U_{1}+e^{i\theta_3}MU_{1}]}\\
	&B\underline{U_2}|_{x_2=0} = 0
	\end{cases}
	\end{aligned}
\end{equation}
which follows from \eqref{eqn:U2underline}.  Similarly, to decouple the problem, assuming $Mr_3 = 0$ gives $\underline{U_2} = 0$ and thus the solution of $\sigma_k^3$.\\

4. Determining the remaining terms of the profile.\\

We consider $U_k$ for $k \geqslant 2$. First from \eqref{eqn:intequation2}, we have:
\begin{equation} 
	\begin{aligned} 
	U_k = E_PU_k + (I - E_PU_k) = \underline{U_k} + \sigma_k^1r_1 + \sigma_k^2r_2 + \sigma_k^3r_3 -R[L(\D)U_{k-1}+e^{i\theta_3}MU_{k-1}]
	\end{aligned}
\end{equation}
We again observe that $(I-E_P)U_k = -R(L(\D)U_{k-1}+e^{i\theta_3}MU_{k-1})$ is determined by previous terms and is known.  Plugging this into $E_Q[L(\D)U_k+e^{i\theta_3}MU_k] = 0$, we have:
\begin{equation} 
	\begin{aligned} 
	E_Q(L(\D)E_PU_k + e^{i\theta_3}ME_PU_k) &= -E_Q(L(\D)(I-E_P)U_k+e^{i\theta_3}M(I-E_P)U_k) \\
	\end{aligned}
\end{equation}
This gives us four equations for $U_k$:
\begin{equation} \label{eqn:Ukunderline}
	\begin{aligned} 
	E_{Q_0}(L(\D)E_PU_k + e^{i\theta_3}ME_PU_k) &= -E_{Q_0}[L(\D)(I-E_P)U_k+e^{i\theta_3}M(I-E_P)U_k] \\
	\end{aligned}
\end{equation}
\begin{equation} \label{eqn:Ukterms}
	\begin{aligned} 
	X_{\phi_1}\sigma_k^1B_2r_1 &= -E_{Q_1}[L(\D)(I-E_P)U_k+e^{i\theta_3}M(I-E_P)U_k]\\
	X_{\phi_2}\sigma_k^2B_2r_2 &= -E_{Q_2}[L(\D)(I-E_P)U_k+e^{i\theta_3}M(I-E_P)U_k]\\
	X_{\phi_3}\sigma_k^3B_2r_3 &+ Q_3e^{i\theta_3}M\sigma_k^3r_3 + 
	Q_3e^{i\theta_3}M\underline{U_k} = 
	\\& -E_{Q_3}[L(\D)(I-E_P)U_k+e^{i\theta_3}M(I-E_P)U_k]
	\end{aligned}
\end{equation}
By the same argument provided for $U_2$,  \eqref{eqn:Ukunderline} implies:
\begin{equation} 
	\begin{aligned} \label{eqn:Ukunderline2}
	L(\D)\underline{U_k} + \underline{e^{i\theta_3}M\sigma_k^3r_3} &= \underline{e^{i\theta_3}MR[L(\D)U_{k-1}+e^{i\theta_3}MU_{k-1}]}
	\end{aligned}
\end{equation}
and the first equation in \eqref{eqn:Ukterms} implies:
\begin{equation}
	\begin{aligned} 
	X_{\phi_1}\sigma_k^1B_2r_1 &= 0\\
	\end{aligned}
\end{equation}
Using $X_{\phi_1}$ is outgoing and $\sigma_k^1 = 0$ when $t < 0$, we have $\sigma_k^1 = 0$.\\
The equation involving $\sigma_k^2$ is also an inhomogenenous transport equation with constant coefficients, thus has unique solutions provided boundary conditions. Boundary condition \eqref{eqn:bry} implies $BE_PU_k + B(I-E_P)U_k = 0$.  Using $B\underline{U_k} = 0$ and isomorphism $B: \text{span} \{r_2,r_3\} \rightarrow \CC^2$, we get:
\begin{equation}
B\sigma_k^2r_2 + B\sigma_k^3r_3 = -B(I-E_P)U_k = c_k^2Br_2 + c_k^3Br_3
\end{equation}
which provides boundary equations $\sigma_k^2 = c_k^2$, $\sigma_k^3 = c_k^3$ and gives the unique solution $\sigma_k^2$. 
\\
The equations involving $\sigma_k^3$ result in the couplings with the following equations involving $\underline{U_k}$:
\begin{equation} 
	\begin{aligned} 
	\begin{cases}
	&X_{\phi_3}\sigma_k^3B_2r_3 + Q_3e^{i\theta_3}M\sigma_k^3r_3 + 
	 Q_3e^{i\theta_3}M\underline{U_k} 
	\\& = -E_{Q_3}[L(\D)(I-E_P)U_k+e^{i\theta_3}M(I-E_P)U_k]\\
	&\sigma_k^3|_{x_2=0, \theta_3 = \theta_0} = c_k^3
	\end{cases}
	\end{aligned}
\end{equation}
\begin{equation} 
	\begin{aligned} 
	&\begin{cases}
	&L(\D)\underline{U_k} + \underline{e^{i\theta_3}M\sigma_k^3r_3} = \underline{e^{i\theta_3}MR[L(\D)U_{k-1}+e^{i\theta_3}MU_{k-1}]}\\
	&B\underline{U_k}|_{x_2 = 0} = 0
	\end{cases}
	\end{aligned}
\end{equation}
which follows from \eqref{eqn:Ukunderline2} and \eqref{eqn:bry}. Similarly, to decouple the problem, assuming $Mr_3 = 0$ gives the solution of $\sigma_k^3$. Now we have determined all unknown in \ref{eqn:profile} and obtain a solution of the desired form.

\subsection{Justifying WKB solutions} \label{justification} \hfill \\

We here justify the solution \ref{eqn:profile} we constructed above by showing that it is close to the real solution as $\eps$ tends to $0$. We write
\begin{equation} \label{eqn:profile'}
u^\eps_a = \sum_{k=1}^N \epsilon^k U_k(t,x,\frac \Phi \eps) 
\end{equation}
as the approximate solution we constructed. Let $u$ be the exact solution to the system. Plugging in $u - u_a^\eps$ to the system, we have:
\begin{equation} \label{3.2}
\begin{aligned}
&L(\D)(u-u_a^\eps) + e^{i\phi_3/\epsilon}M(u-u_a^\eps) = -\eps^N[(L(\D)+e^{i\phi_3/\epsilon}M)U_N(t,x,\theta)]|_{\theta=\Phi/\epsilon}\\
&B(u-u_a^\eps)= 0 \text{ on boundary}
\end{aligned}
\end{equation}
where $L(\D) := \D_t + B_1\D_{x_1} + B_2\D_{x_2}$. 

We need the following proposition:
\begin{prop} \label{prop3.3}
Let $f$ be periodic in $\theta$ and $f(t,x,\theta)\in H^m(t,x,\theta)$ for $m\geq \frac{n+1}{2}$ $((t,x)\in \RR^n)$, then $|f(t,x,\frac{\phi(t,x)}{\epsilon})|_{L^2(t,x)} \leq C|f(t,x,\theta)|_{H^1(t,x,\theta)}$
\end{prop}
\begin{proof}
Writing
\begin{equation}
f(t,x,\theta) = \sum_1^\infty f_n(t,x)e^{i\theta/\eps}, \quad f(t,x,\frac{\phi(t,x)}{\eps}) = \sum_1^\infty f_n(t,x)e^{i\phi(t,x)/\eps}
\end{equation}
we have
\begin{align}
\left|\sum_1^\infty f_n(t,x)e^{i\phi(t,x)/\eps} \right|_{L^2(t,x)} &\leq \sum_1^\infty |f_n(t,x)|_{L^2(t,x)}
\\&\leq \left|\left(|f_n(t,x)|_{L^2(t,x)}\cdot n\right)\right|_{\ell^2(n)}\left| \frac{1}{n}\right|_{\ell^2(n)}
\\&\leq C|f(t,x,\theta)|_{H^1(t,x,\theta)}
\end{align}
\end{proof}
We need the following estimate from \cite{Kre70}:


\begin{prop} \cite{Kre70}
There exists a constant $\gamma_0 $ such that for all $\gamma \geq \gamma_0$, the system
\begin{align}
&L(\D)u = f\\
&Bu=g
\end{align}
satisfies:
\begin{equation} \label{cg10estimate}
\gamma |e^{-\gamma t}u|_{L^2(t,x)}^2 + |e^{-\gamma t}u|_{x_2 = 0}|_{L^2(t,x_1)}^2 \leq C\left(\frac{1}{\gamma}|e^{-\gamma t}f|^2_{L^2(t,x)}+ |e^{-\gamma t}g|^2_{L^2(t,x_1)}\right).
\end{equation}
\end{prop}
\begin{prop}
The estimate \eqref{cg10estimate} holds for the system:
\begin{align}
&L(\D)u + \mathcal{D}u = f\\
&Bu=g
\end{align}
where $\mathcal{D}$ is any bounded function.
\end{prop}
\begin{proof}
We can rewrite the first equation as $L(\partial)u  = f -  \mathcal{D}u$. Then applying \eqref{cg10estimate} to $(L(\D),B)$ gives:
\begin{align}
\gamma |e^{-\gamma t}u|_{L^2(t,x)}^2 + &|e^{-\gamma t}u|_{x_2 = 0}|_{L^2(t,x_1)}^2 \leq C\left(\frac{1}{\gamma}|e^{-\gamma t}(f -  \mathcal{D})u|^2_{L^2(t,x)}+ |e^{-\gamma t}g|^2_{L^2(t,x_1)}\right)
\\&\leq C\left(\frac{1}{\gamma}|e^{-\gamma t} f|^2_{L^2(t,x)}+\frac{1}{\gamma}|\mathcal{D}|_{L^\infty} |e^{-\gamma t}  u|^2_{L^2(t,x)}+ |e^{-\gamma t}g|^2_{L^2(t,x_1)}\right).
\end{align}
Taking $\gamma$ large gives the desired inequality.
\end{proof}
\begin{prop}
There exists a unique solution for the system \eqref{eqn:system}.
\end{prop}
\begin{proof}
Consider the following system:
\begin{equation} \label{wsystem}
\begin{aligned}
L(\D)w^\eps + e^{i\phi_3/\epsilon}Mw^\eps &= -\eps^N[(L(\D)+e^{i\phi_3/\epsilon}M)U_N(t,x,\theta)]|_{\theta=\Phi/\epsilon} =: f\\
Bw^\eps &= 0 \text{ on boundary}
\end{aligned}
\end{equation}
We observe that if there exists $w^\eps$ satisfying \eqref{wsystem}, then  $u^\eps = u_a^\eps + w^\eps$ solves equation \eqref{eqn:1system}. Therefore, it is left to prove the existence and uniqueness of $w^\eps$.  Consider the dual problem:
\begin{align} \label{dual}
L^*(\D)v &= f_\sharp\\
A_\sharp v &= g_\sharp
\end{align}
where $L^*(\D)$ is the adjoint operator of $L(\D) + e^{i\phi_3/\epsilon}M$, $f_\sharp, g_\sharp$ are appropriate functions and $A_\sharp$ satisfies:
\begin{equation} \label{ABsharp}
B_\sharp^T B + A_\sharp^T A = I.
\end{equation}
Define the space
\begin{equation}
F = \{V \in C^\infty_0 (\overline{\Omega}) \text{ such that } A_\sharp V|_{x_2 = 0} = 0\}
\end{equation}
For all $V\in F$, define 
\begin{equation} \label{ell}
\ell[L^*(\D)V]:= (f,V)_{L^2} + (0, B_\sharp V|_{x_2 = 0})_{L^2(t,x_1)}=(f,V)_{L^2}.
\end{equation} 
By definition $\ell$ is linear.  We observe that the system \eqref{dual} is the same form as \eqref{eqn:1system}, and therefore satisfies the same energy estimate \eqref{cg10estimate}. So we have:
\begin{align}
\ell[L^*(\D)V] & \leq \|f\|_{L^2(t,x)} \|V\|_{L^2(t,x)} \\
& \leq C \|V\|_{L^2(t,x)} \\
& \leq C\|L^*(\D)V\|_{L^2(t,x)} 
\end{align}
This shows that $\ell$ is continuous. Now by Hahn-Banach theorem we extend $\ell$ to $L^2(\Omega_T)$. Riesz representation theorem implies that there exists $W \in L^2(\Omega_T)$ such that 
\begin{equation} \label{riesz}
\ell[L^*(\D)V] = (W,L^*(\D)V)_{L^2(t,x)} =  (L(\D)W,V)_{L^2(t,x)}
\end{equation}
Thus, we have $L(\D)W = f$. Moreover,  the trace of $W$ lies in $H^{-1/2}(t,x_1)$, and the following Green's formula holds:
\begin{equation} \label{greens}
(L(\D)W,V)_{L^2(t,x)} = (f,V)_{L^2(t,x)} + (W, V)_{H^{-1/2}(t,x_1)H^{1/2}(t,x_1)} \quad \forall V \subset C_0^\infty(\overline{\Omega}).
\end{equation}
Comparing \eqref{greens} and \eqref{ell} gives:
\begin{equation}
(W(t,x_1,0),V(t,x_1,0))_{L^2(t,x_1)} = (0, B_\sharp V(t,x_1,0))_{L^2(t,x_1)} = 0 \quad \forall V\in F
\end{equation}
We can therefore conclude that $W = 0$ on $x_2 = 0$. This concludes the proof of the proposition.
\end{proof}

Using proposition \ref{prop3.3} and estimate \eqref{cg10estimate}, we have
\begin{equation} \label{3.3}
\begin{aligned}
\gamma |e^{-\gamma t}(u-u_a^\eps)|^2_{L^2(t,x)} & \leq
\frac{C}{\gamma} |e^{-\gamma t}\eps^N([L(\D)+e^{\frac{\phi_3(t,x)}{\epsilon}}M)U_N(t,x,\theta)]|_{\theta=\Phi/\epsilon}|^2_{L^2(t,x)} \\
|e^{-\gamma t}(u-u_a^\eps)|_{L^2(t,x)} &\leq  \frac{C}{\gamma} \eps^{N} |e^{-\gamma t}[(L(\D)+e^{i\phi_3/\epsilon}M)U_N(t,x,\theta)]|_{\theta=\Phi/\epsilon}|_{L^2(t,x)} \\
&\leq \frac{C}{\gamma} \eps^{N} |(L(\D)+e^{i\theta_3}M)U_N(t,x,\theta)|_{H^1(t,x,\theta)}
\end{aligned}
\end{equation}
Since $e^{-\gamma T}|(u-u_a^\eps)|_{L^2(t,x)} \leq |e^{-\gamma t}(u-u_a^\eps)|_{L^2(t,x)}$ for $t \in [0,T]$,  we have:
\begin{equation}
|(u-u_a^\eps)|_{L^2(t,x)}
\leq \frac{Ce^{\gamma T}}{\gamma} \eps^{N} |(L(\D)+e^{i\theta_3}M)U_N(t,x,\theta)|_{H^1(t,x,\theta)}
\end{equation}
which gives:
\begin{equation} \label{uua}
|(u-u_a^\eps)|_{L^2(t,x)} \leq C \epsilon^{N}
\end{equation}
To attain higher order derivatives, we commute the system \ref{3.2} with tangential derivatives $\epsilon^{|\alpha|} \D^\alpha_{t,x_1}$ ($|\alpha| \leq s$, $s > n/2$).
We first commute the system with $\epsilon \D_t$:
\begin{equation}
\begin{aligned}
&\epsilon \D_tL(\D)(u-u_a^\eps) + \epsilon \D_t(e^{i\phi_3/\epsilon}M(u-u_a^\eps)) = -\eps^N[\epsilon \D_t(L(\D)+e^{i\phi_3/\epsilon}M)U_N(t,x,\theta)]|_{\theta=\Phi/\epsilon}\\
&\epsilon \D_tB(u-u_a^\eps)= 0 \text{ on boundary}
\end{aligned}
\end{equation}
After applying chain rule and reorganizing terms, we have $\D_t(u-u_a^\eps)$ satisfying:
\begin{equation} \label{3.8}
\begin{aligned}
 L(\D)(\D_t(u-u_a^\eps)) + e^{i\phi_3/\epsilon}M(\D_t(u-u_a^\eps)) &= -\eps^{N} \D_t[(L(\D)+e^{i\phi_3/\epsilon}M)U_N(t,x,\theta)]|_{\theta=\Phi/\epsilon} 
\\&- \frac{1}{\eps}i e^{i\phi_3/\eps}M\D_t\phi_3(u-u_a^\eps)\\
B(\D_t(u-u_a^\eps)) &= 0 \text{ on boundary}
\end{aligned}
\end{equation}
where the right hand side of the first equation is $\mathcal{O}(\eps^{N-1})$ by \ref{uua}. Thus, applying the argument \ref{3.3} to \ref{uua} to the system \ref{3.8} gives:
\begin{equation} 
|\D_t(u-u_a^\eps)|_{L^2(t,x)} \leq C \epsilon^{N-1}
\end{equation}
We note that commuting \ref{3.2} with $\eps\D_{x_1}$ shows that $\D_{x_1}(u-u_a^\eps)$ satisfies the exact system as \ref{3.8} with $\D_{t}$ replaced by $\D_{x_1}$, thus :
\begin{equation} 
|\D_{x_1}(u-u_a^\eps)|_{L^2(t,x)} \leq C \epsilon^{N-1}
\end{equation}
We iterate the above argument and get estimate for higher derivatives:
\begin{equation}
|\D^{|\alpha|}_{t,x_1}(u-u_a^\eps)|_{L^2(t,x)} \leq C \epsilon^{N-|\alpha|}
\end{equation}
Finally, to estimate the $\D_{x_2}(u-u_a^\epsilon)$, using $B_2$ is invertible, we have from the interior equation:
\begin{equation} \label{dx2}
\begin{aligned}
&(\D_t + B_1\D_{x_1} + B_2\D_{x_2})(u-u_a^\eps) + e^{i\phi_3/\epsilon}(u-u_a^\eps) = -\eps^N(L(\D)+e^{i\phi_3/\epsilon}M)U_N \\
\Rightarrow &B_2\D_{x_2}(u-u_a^\eps) = -\eps^N(L(\D)+e^{i\phi_3/\epsilon}M)U_N - \D_t (u-u_a^\eps) - B_1\D_{x_1}(u-u_a^\eps) - e^{i\phi_3/\epsilon}(u-u_a^\eps) \\
\Rightarrow &\D_{x_2}(u-u_a^\eps) = B_2^{-1}[-\eps^N(L(\D)+e^{i\phi_3/\epsilon}M)U_N - \D_t (u-u_a^\eps) - B_1\D_{x_1}(u-u_a^\eps) - e^{i\phi_3/\epsilon}(u-u_a^\eps)]
\end{aligned}
\end{equation}
where we already have an estimate for the all the terms that appear on the right, which gives:
\begin{equation}
|\D_{x_2}(u-u_a^\eps)|_{L^2(t,x)} \leq C\epsilon^{N-1}.
\end{equation}
Now we iterate this process by continue applying $\D_{x_2}$ to \eqref{dx2}. This gives estimates to higher order $x_2$ derivative:
\begin{equation}
|\D^\alpha_{x_2}(u-u_a^\eps)|_{L^2(t,x)} \leq C\epsilon^{N-|\alpha|}
\end{equation}
Now, for $s = |\alpha| > \frac{n}{2}$, ($n=3$ here) we have:
\begin{equation}
|u-u_a^\eps|_{L^{\infty}(t,x)} \leq C|u-u_a^\eps|_{H^s(t,x)} = C|\D^\alpha(u-u_a^\eps)|_{L^2(t,x)} \leq C\epsilon^{N-|\alpha|}
\end{equation}
Now we recall that from \eqref{eqn:profile'}, we can compute as many terms as we want in the expansion of $u_a^\eps$. Therefore, for any $n \in \NN$, we can choose the highest order of approximate solution $N \geq n+2$ to  conclude that:
\begin{equation}
|u-u_a^\eps|_{L^{\infty}(t,x)} \leq C\epsilon^n
\end{equation}
\newpage

\section{Estimations for Weakly Stable Hyperbolic System} \label{estimatations}


 In this section we study transformed singular system \eqref{i9}. We assume that  the uniform Lopatinski condition \textbf{fails} only on the set $\Upsilon_0 = \{\beta_l, -\beta_l\}\subset \mathcal{H}$,  and that the oscillatory term may have both positive and negative Fourier spectrum.  In the case where there is only one incoming mode,  \cite{W4} Theorem 2.12 gives the following estimate  to the system \eqref{i9}.
 \begin{equation}
 |U^\gamma|_{L^2(t,x,\theta)} + \left|\frac{U^\gamma(0)}{\sqrt\gamma} \right|_{L^2(t,x_1,\theta)} \leq K \left[\frac{1}{\gamma^2}\left(\sum_{k\in\ZZ} \left||X_k|\widehat{F_k}\right|^2_{L^2}\right)^{1/2} + \frac{1}{\gamma^{3/2}}\left(\sum_{k\in\ZZ} \left||X_k|\widehat{G_k}\right|^2_{L^2(\sigma,\eta)}\right)\right]
 \end{equation}
 
In the case where there are at least two incoming modes, there is so far no analogous theorem,  but we have been able to prove such a result in the small/medium frequency region (Theorem \ref{mainthm}):
\begin{equation}
|\chi_D U^\gamma|_{L^2(t,x,\theta)} + \left|\frac{\chi_D U^\gamma (0)}{\sqrt{\gamma}}\right|_{L^2(t,x_1,\theta)} \leq 
K\left[ \frac{1}{\gamma^2}(\sum_{k\in\ZZ}\left| \chi |X_k|\widehat{F_k}\right|^2_{L^2(x_2,\sigma,\eta)})^{1/2} + \frac{1}{\gamma^{3/2}}(\sum_{k\in\ZZ}\left| \chi|X_k|\widehat{G_k}\right|^2_{L^2(\sigma,\eta)})^{1/2}\right]
\end{equation}
where $\chi_D$ and $\chi$ are the characteristic functions corresponds to the small/medium frequency region $|\zeta|\leq \eps^{\alpha-1}$, $0<\alpha<1$. 
The first step is to prove a  ``small/medium frequency iteration estimate" (proposition \ref{tt30}), and that is then used to prove  the main result.

\subsection{Extensions to $\Gamma_\delta$ and then to $\Xi$} \cite{W4} \hfill \\

We here extend the variables defined in section \ref{sec2} in the same way as \cite{W4}.  We recall that the $\Gamma^+_\delta(\beta_l)$ is a conic neighborhood of $\beta_l$, next we define the ``opposite conic" $\Gamma^-_\delta(\beta_l)$:
\begin{align}\label{ss7}
\Gamma^-_\delta(\beta_l):=\{(\sigma-i\gamma,\eta):(-\sigma-i\gamma,-\eta)\in\Gamma^+_\delta(\beta_l)\}.
\end{align}
We extend the eigenvalues $\omega_j$ and $R_j$ to the $\Gamma_\delta(\beta_l) : =\Gamma^+_\delta(\beta_l) \cup \Gamma^-_\delta(\beta_l)$ by the following equations:
\begin{align}\label{s7ff}
R_j(\sigma-i\gamma,\eta)=\overline{R}_j(-\sigma-i\gamma,-\eta) \text{ for }(\sigma-i\gamma,\eta)\in\Gamma^-_\delta(\beta_l).\\
\omega_j(\sigma-i\gamma,\eta)=\overline{\omega}_j(-\sigma-i\gamma,-\eta) \text{ for }(\sigma-i\gamma,\eta)\in\Gamma^-_\delta(\beta_l).
\end{align}
This extension keeps $R_j$ analytic in $\tau$, $C^\infty$ in $\eta$,   and homogeneous of degree $1$ in $\Gamma_\delta(\beta_l)$, and $\omega_j$  analytic in $\tau$.  We define the new normalized $r_j := \frac{R_j}{|R_j|}$ using the extended $R_j$. This also gives an extension of the matrix $S(\zeta)$ defined in \eqref{szeta}.

We extend $R_j$ to $\Xi$ in a way such that $S(\zeta)$ is invertible and $S^{-1}$ is uniformly bounded.  We extent $\omega_j$ to $\Xi$ such that the extended $\omega_j$ are $C^\infty$, analytic in $\tau$, and homogeneous of degree $1$.

We observe that $\mathcal{A}$ is real since each $A_j$ are,  so we have
\begin{align}\label{s7fk}
\mathcal{A}(\zeta)R_j(\zeta)=\omega_j(\zeta)R_j(\zeta)\text{ on }\Gamma_\delta(\beta_l).
\end{align}
By the way of extension \eqref{s7ff}, the following inequality still hods on $\Gamma_\delta(\beta_l)$:
 \begin{align}\label{s7fi}
 \begin{split}
 &\mathrm{Im}\;\omega_j(\zeta)\leq -c\gamma \text{ for }j\in\mathcal{O}\\
 &\mathrm{Im}\;\omega_j(\zeta)\geq c\gamma \text{ for }j\in \mathcal{I},
\end{split}
\end{align}
for the same constant $c$ in \eqref{s5}.    This implies that on $\Gamma_\delta(\beta_l)$, $R_j(\zeta)\in \mathbb{E}^s(\zeta)$ (definition \ref{stablesubspace}) for $j\in \mathcal{I}$.  We also have $\det BR_-(\pm\beta_l)=0$ since $R_j$s are real. \footnote{The extensions in this section can be applied in the conic neighborhood of any vector and doesn't rely on the failure of uniform Lopatinski condition. }  

\subsection{Tools for Estimates} \label{sectools} \cite{W4}\hfill \\

This section quoted earlier results and establish definitions to be used in proving a uniform estimate.  This section is established by Williams \cite{W4} but with modified definitions \ref{t28} and \ref{t29}.
\begin{prop}\label{g1z}
\cite{W4} For $\gamma>0$, $\tau=\sigma-i\gamma$ we  have on $x_2\geq 0$:
\begin{align}\label{g1}
\begin{split}
&(a) \left|\int^{x_2}_0e^{-\gamma (x_2-s)}f(s,\tau,\eta)ds\right|_{L^2(x_2,\sigma,\eta)}\leq \frac{1}{\gamma}|f|_{L^2(x_2,\sigma,\eta)}\\
&(b) \left|\int^{\infty}_{x_2}e^{\gamma (x_2-s)}f(s,\tau,\eta)ds\right|_{L^2(x_2,\sigma,\eta)}\leq \frac{1}{\gamma}|f|_{L^2(x_2,\sigma,\eta)}\\
&(c)   \left|\int^{\infty}_0e^{-\gamma s}f(s,\tau,\eta)ds\right|_{L^2(\sigma,\eta)}\leq \frac{1}{\sqrt{2\gamma}}|f|_{L^2(x_2,\sigma,\eta)}\\
&(d) \left|e^{-\gamma x_2}g(\tau,\eta)\right|_{L^2(x_2,\sigma,\eta)} = \frac{1}{\sqrt{2\gamma}}|g|_{L^2(\sigma,\eta)}.
\end{split}
\end{align}

\end{prop}

\begin{notn} \label{notation}
\begin{enumerate}
\item For $\zeta = (\tau,\eta) \in \Xi$,  we define $X_k := \zeta+k\frac{\beta_l}{\eps}$. 
\item For any function $f(\zeta)$,  we write:
\begin{equation}
f(\eps,k) = f(\eps,k)(\zeta) := f(X_k)
\end{equation}
\item We define $\chi_b(\zeta)$ be the characteristic function of $\Gamma_\delta(\beta_l)$. Thus, $\zeta \in \text{supp} \chi_b(\eps,k)$ if and only if $X_k \in \Gamma_{\delta}(\beta_l)$.
\end{enumerate}
\end{notn}

An important tool we use is the following lemma from \cite{W4}, which is an easy consequence of assumption \ref{lopassump} and relevant definitions. 
\begin{lem}\label{t6}\cite{W4}
For $\beta_l \in\Upsilon^+_0$.  Recall that $X_k:=\zeta+k\frac{\beta_l}{\eps}$ and that $\zeta\in \mathrm{supp}\;\chi_b(\eps,k)\Leftrightarrow X_k\in \Gamma_\delta(\beta_l)$. 
For $k\in \ZZ$ the following estimates hold:
\begin{align}\label{t6a}
\begin{split}
&(a)\;|\Delta(\eps,k)|\lesssim 1\text{ on }\Xi \\
&(b)\; \left|[Br_-(\eps,k)]^{-1}\right|\lesssim |\Delta(\eps,k)|^{-1}\sim \frac{|X_k|}{|\tau-c_+(\beta)\eta|}\text{ on }\mathrm{supp}\;\chi_b(\eps,k)\\
&(c)\; \left| [Br_-(\eps,k)]^{-1}\right|\lesssim |\Delta(\eps,k)|^{-1} \leq C(\delta)\text{ on }\Xi\setminus  \mathrm{supp}\;\chi_b(\eps,k) \\
&(d)\; |Br_{\pm}(\eps,k)|\lesssim 1 \text{ on }\Xi,\\
& (e)\; |\omega_i(\eps,k)-\omega_j(\eps,k)|\sim |X_k| \text{ for }i\neq j\text{ on }\mathrm{supp}\;\chi_b(\eps,k)\\
&(f)\; \mathrm{Im}\;\omega_j(\eps,k)\leq -c\gamma\text{ for }j\in \mathcal{O},\;\;\;\mathrm{Im}\;\omega_j(\eps,k)\geq c\gamma\text{ for }j\in \mathcal{I}, \text{ on }\mathrm{supp}\;\chi_b(\eps,k).
\end{split}
\end{align}

(g) Let $r\in\ZZ\setminus 0$.  When $X_k \notin  \Gamma_{\frac{\delta}{|r|}}(\beta)$ we have  $|\Delta(\eps,k)|^{-1}\leq C(\delta)|r|$.

\end{lem}

We also have the following lemma:
\begin{lem}\label{c9a}\cite{W4}
 Let $X_{k}=\zeta+k\frac{\beta_l}{\eps}$, $X_{k-r}=\zeta+(k-r)\frac{\beta_l}{\eps}$, where $r\in\ZZ\setminus 0$.    For $\delta\in (0,\delta_0]$ and $N_1\in \NN$ sufficiently large, 
 assume that 
 \begin{align}\label{c10}
 X_{k}\in\Gamma_{\frac{\delta}{N_1|r|}}(\beta_l), \text{ but } X_{k-r}\notin\Gamma_{\frac{\delta}{|r|}}(\beta_l).
 \end{align}
  Then 
 $|X_{k-r}|\lesssim \frac{1}{N_1}|X_k|$.
 \end{lem}
 
We next define $E_{i,j}$ which is one main tool to be used to control the amplification of $\Delta^{-1}$.
\begin{defn} \label{defneij}\cite{W4}
Let $k\in \ZZ$, $r\in \ZZ \setminus 0$. For $\zeta = (\tau,\eta)\in \Xi$ we define the function of $\zeta$:
\begin{equation}
E_{i,j}(\eps,k,k-r)(\zeta):= \omega_i(\eps,k) - \frac{r\omega_N(\beta_l)}{\eps}-\omega_j(\eps,k-r), \text{ where } i\in \mathcal{O}, j\in \mathcal{I}
\end{equation}
\end{defn}

We separately study the following three cases for the pair $X_k=\zeta+k\frac{\beta_l}{\eps}$, $X_{k-r}=\zeta+(k-r)\frac{\beta_l}{\eps}$ for $\delta$ small and $N_1$ large:
\begin{align}\label{cases}
\begin{split}
& (I) \;X_k\in\Gamma_{\frac{\delta}{N_1|r|}}(\beta_l),  X_{k-r}\in\Gamma_{\frac{\delta}{|r|}}(\beta_l)\\
& (II) \;X_k\in\Gamma_{\frac{\delta}{N_1|r|}}(\beta_l),  X_{k-r}\notin\Gamma_{\frac{\delta}{|r|}}(\beta_l)\\
& (III) \;X_k\in\Gamma_\delta(\beta_l)\setminus \Gamma_{\frac{\delta}{N_1|r|}}(\beta_l)
\end{split}
\end{align}

We proceed to define a global and local amplification factor which is used to control the $\frac{1}{\eps}$ term. We'll see later how the proof of iteration estimate gives rise to this term.

\begin{defn}\label{t28}[Microlocal amplification factors]
Let $C_5\geq 1$, $N_1$ be sufficiently large constants,  and let $0<\xi<1$ denote a number to be chosen (in proposition \ref{finiter}), 

For $k\in\ZZ$, $r\in \ZZ\setminus 0$, and $(\zeta,\eps)\in \mathrm{supp}\;\chi_b(\eps,k) \times (0,\eps_0]$, we define:

$\bullet$ $D(\eps,k,k-r)(\zeta)=C_5|r|$ in case $(III)$

 $\bullet$ $D(\eps,k,k-r)(\zeta)=C_5|r|$ in case $(II)$.

 \noindent If case $(I)$ obtains, then define
 
 $\bullet$ $D(\eps,k,k-r)(\zeta)=\begin{cases} C_5|r|^{2+\delta}, \text{ when } |r| \leq \eps^{-\xi}\\\frac{C_5|r|}{\eps\gamma}\text{ otherwise} \end{cases}$

\noindent If $\zeta\notin \mathrm{supp}\;\chi_b(\eps,k)$, define $D(\eps,k,k-r)(\zeta)=0$.
\end{defn}

\begin{defn}\label{t29}[Global amplification factors]
For $k\in\ZZ$, $r\in\ZZ\setminus 0$,  and $(\zeta,\eps)\in \mathrm{supp}\;\chi_b(\eps,k) \times (0,\eps_0]$, we define
\begin{align}\label{ty29}
\begin{split}
&\DD(\eps,k,k-r)(\zeta)=\begin{cases}\frac{C_5|r|}{\eps\gamma},\text{ if }D(\eps,k,k-r)(\zeta)=\frac{C_5|r|}{\eps\gamma}\\C_5|r|^{2+\delta}, \; otherwise\end{cases}.
\end{split}
\end{align}
\noindent If $\zeta\notin \mathrm{supp}\;\chi_b(\eps,k)$  we set $\DD(\eps,k,k-r)(\zeta)=1$.

\end{defn}
We observe that by definition 
\begin{align}\label{t30}
D(\eps,k,k-r)(\zeta)\leq \DD(\eps,k,k-r)\text{ for all }\zeta\in\Xi
\end{align}

Recall that $V_k:= \widehat{U}_k(\zeta,x_2)$, the Laplace-Fourier transform in $(t,x_1)$ of $U_k(t,x)$.  Let $\chi_g$ be the characteristic function of $\Gamma_\delta(\beta_l)^c$ (the compliment of $\Gamma_\delta(\beta_l)$). We can write using notation \ref{notation}:
\begin{equation}
V_k(x_2,\zeta) = \chi_b(\eps,k)V_k + \chi_g(\eps,k)V_k.
\end{equation}
Define $w_k = (w_k^+, w_k^-)$ for all $\zeta \in \Xi$ by
\begin{equation}
V_k(x_2,\zeta) = S(\eps,k)w_k(x_2,\zeta)
\end{equation} where $S(\zeta)$ as in \eqref{szeta}. For $\zeta\in  \text{supp} \chi_b(\eps,k)$,  by definitions and obvious transforming \eqref{i9}, we see that $w_k$ is a solution of the diagonalized system:
  \begin{align}\label{a3a}
\begin{split}
&D_{x_2}w_k-\begin{pmatrix}\xi_+(\eps,k)&0\\0&\xi_-(\eps,k)\end{pmatrix}w_k=\\
&\qquad i\sum_{r\in\ZZ\setminus 0}\alpha_r e^{ir\frac{\omega_N(\beta_l)}{\eps}x_2}S^{-1}(\eps,k)B_2^{-1}M S(\eps,k-r)w_{k-r}+S^{-1}(\eps,k;\beta)\widehat{F^\eps_k}(x_2,\zeta),\\
&BS(\eps,k)w_k= \hat G_k\text{ on }x_2=0.
\end{split}
\end{align}
where 
 \begin{align}\label{a3b}
 \begin{split}
 &\xi_+(\eps,k)=\mathrm{diag}\;(\omega_1(\eps,k),...,\omega_{N-p}(\eps,k))\\
 &\xi_-(\eps,k)=\mathrm{diag}\;(\omega_{N-p+1}(\eps,k),...,\omega_{N}(\eps,k)).
\end{split}
\end{align}
Solutions to \eqref{a3a} for $F=0$ satisfy:
\begin{align}\label{a5}
w^+_k(x_2,\zeta)=\sum_{r\in\ZZ\setminus 0}\int^\infty_{x_2}e^{i\xi_+(\eps,k)(x_2-s)+ir\frac{\omega_N(\beta_l)}{\eps}s}\ar[a(\eps,k,k-r)w^+_{k-r}(s,\zeta)+b(\eps,k,k-r)w^-_{k-r}(s,\zeta)]ds,
\end{align}
\begin{align}\label{a6}
\begin{split}
&w^-_k(x_2,\zeta)=-\sum_{r\in\ZZ\setminus 0}\int^{x_2}_0e^{i\xi_-(\eps,k)(x_2-s)+ir\frac{\omega_N(\beta_l)}{\eps}s}\ar[c(\eps,k,k-r)w^+_{k-r}(s,\zeta)+d(\eps,k,k-r)w^-_{k-r}(s,\zeta)]ds-\\
&e^{i\xi_-(\eps,k)x_2}[Br_-(\eps,k)]^{-1}Br_+(\eps,k)\sum_{r\in\ZZ\setminus 0}\int^\infty_{0}e^{i\xi_+(\eps,k)(-s)+ir\frac{\omega_N(\beta_l)}{\eps}s}\ar[a(\eps,k,k-r)w^+_{k-r}(s,\zeta)+\\
&\qquad\qquad\qquad b(\eps,k,k-r)w^-_{k-r}(s,\zeta)]ds+e^{i\xi_-(\eps,k)x_2}[Br_-(\eps,k)]^{-1}\hat G_k(\zeta).
\end{split}
\end{align}
The matrices $a,b,c,d$ in \eqref{a5}, \eqref{a6} are block matrix such that:
\begin{align}
\begin{pmatrix}
a&b \\
c&d 
\end{pmatrix} = S(\eps,k)^{-1}B_2^{-1}MS(\eps,k-r)
\end{align}
and we have
\begin{align}\label{a7}
|a(\eps,k,k-r)|\lesssim 1, |b(\eps,k,k-r)|\lesssim 1, |c(\eps,k,k-r)|\lesssim 1, |d(\eps,k,k-r)|\lesssim 1
\end{align}

We also have the following form of equations for $w^\pm_k$:
\begin{align}\label{ab7}
w^+_k(x_2,\zeta)=\sum_{r\in\ZZ\setminus 0}\int^\infty_{x_2}e^{i\xi_+(\eps,k)(x_2-s)+ir\frac{\omega_N(\beta_l)}{\eps}s}\ar M^+(\eps,k,k-r)V_{k-r}(s,\zeta)ds
\end{align}
  \begin{align}\label{ab8}
\begin{split}
&w^-_k(x_2,\zeta)=-\sum_{r\in\ZZ\setminus 0}\int^{x_2}_0e^{i\xi_-(\eps,k)(x_2-s)+ir\frac{\omega_N(\beta_l)}{\eps}s}\ar M^-(\eps,k,k-r)V_{k-r}(s,\zeta)ds-\\
&e^{i\xi_-(\eps,k)x_2}[Br_-(\eps,k)]^{-1}Br_+(\eps,k)\sum_{r\in\ZZ\setminus 0}\int^\infty_{0}e^{i\xi_+(\eps,k)(-s)+ir\frac{\omega_N(\beta_l)}{\eps}s}\ar M^+(\eps,k,k-r)V_{k-r}(s,\zeta)ds+\\
&\qquad \qquad e^{i\xi_-(\eps,k)x_2}[Br_-(\eps,k)]^{-1}\hat G_k(\zeta),
\end{split}
\end{align}
  where the definitions of matrices $M^\pm$ are obvious and $|M^\pm(\eps,k,k-r)|\lesssim 1$. \\

Next, we define 
\begin{align}\label{tu29}
\mathcal{W}_k(x_2,\zeta)=(\tilde w^+_k(x_2,\zeta),w^-_k(x_2,\zeta),\frac{1}{\sqrt{\gamma}}\tilde w^+_k(0,\zeta),\frac{1}{\sqrt{\gamma}}w^-_k(0,\zeta)),
\end{align}
where
\begin{align}
\tilde w^+_k(x_2,\zeta)=\Delta^{-1}(\eps,k)w^+_k(x_2,\zeta).
\end{align}
For each $k$ we  define a modified $L^2$ norm of $V_k$ 
 by\footnote{In \eqref{tu30} the notation $|\cdot|_{L^2}$ means $|\cdot|_{L^2(x_2,\sigma,\eta)}$ for components that depend on $x_2$ and  $|\cdot|_{L^2(\sigma,\eta)}$ for components that do not.}
\begin{align}\label{tu30}
\|V_k\|_k=|\chi_b(\eps,k)\CalW_k(x_2,\zeta)|_{L^2}+\left|\chi_g(\eps,k)\left(V_k(x_2,\zeta),\frac{1}{\sqrt{\gamma}}V_k(0,\zeta)\right)\right|_{L^2}, 
\end{align}
We note that we ultimately want to estimate $|(\|V_k\|)|_{\ell^2(k)}$ of $(V_k)_{k\in\ZZ}$ (theorem \ref{mainthm}).  We usually write $\|V_k\|$  $\|V_k\|_k$.    If $f(\zeta)$ is any function of $\zeta$ then
\begin{align}\label{tuu30}
\|f(\zeta)V_k\|_k:=|\chi_b(\eps,k)f(\zeta)\CalW_k(x_2,\zeta)|_{L^2}+\left|\chi_g(\eps,k)f(\zeta)\left(V_k(x_2,\zeta),\frac{1}{\sqrt{\gamma}}V_k(0,\zeta)\right)\right|_{L^2}.
\end{align}

Note that for $U^\gamma$ as in Theorem \ref{mainthm}
\begin{align}\label{tu31}
|(\|V_k\|)|_{\ell^2(k)}\gtrsim |U^\gamma|_{L^2(t,x,\theta)}+\left|\frac{U^\gamma(0)}{\sqrt{\gamma}}\right|_{L^2(t,x_1,\theta)}.
\end{align}

\subsection{Analysis of Small/Medium Frequency Region} \hfill \\

Fix $0 <\alpha <1$, we now we consider the effect of $\zeta$ small such that  
\begin{equation}\label{f14a}
|\zeta| \lesssim \eps^{\alpha -1} \quad \text{for } 0<\alpha<1
\end{equation}
We recall (definition \ref{defnres}) that the phases $(\phi_j,\phi_N,\phi_i)$ exhibit a resonance if there exist $p,q \in \ZZ \setminus 0$ such that:
\begin{equation}
p\phi_j + q\phi_N = (p+q)\phi_i \iff p\omega_j(\beta_l) + q \omega_N(\beta_l)= (p+q)\omega_i(\beta_l) \iff \frac{p}{q} = \frac{\omega_i(\beta_l)-\omega_j(\beta_l)}{\omega_j(\beta_l)-\omega_N(\beta_l)} = \Omega_{i,j}
\end{equation}
Proposition 4.6 in \cite{W4} showed that it is impossible to have the form of estimates in \eqref{5.12} or \eqref{rdependoneps} if there exists a resonance. This motivates us to make an assumption \ref{omega} and study the case where $p/q$ is away from $\Omega_{i,j}$.

\begin{prop} 
\begin{enumerate}
\item {Thue-Siegel-Roth theorem.}
If $\Omega$ is an algebraic number, then there exists a positive constant $C(x,\delta)$ such that for any $p,q \in \ZZ$, $q \neq 0$, we have:
\begin{align} \label{thsirothm}
\left| \frac{p}{q} - \Omega \right| > \frac{C(x,\delta)}{|q|^{2+\delta}} \quad \delta > 0
\end{align}
\item There is a full measure set of irrational numbers such that \eqref{thsirothm} holds.
\end{enumerate} 
\end{prop}

\begin{assump} \label{omega}
$\forall i \in \mathcal{O}, j \in \mathcal{I}\setminus\{N\}$, the irrational number $\Omega_{i,j} = \frac{\omega_i(\beta_l)-\omega_j(\beta_l)}{\omega_j(\beta_l)-\omega_N(\beta_l)}$  (definition \eqref{Omegadefn}) satisfies \eqref{thsirothm}.
\end{assump}


The following lemma is clear.
\begin{lem} \label{epslemma}
For $A,B \in \RR$, there exists $\eps_0$ small such that for $0< \eps < \eps_0$, we have $A\eps^{-\alpha} + B\eps^{-\beta} \sim \eps^{-\beta}$ if $0 <\alpha < \beta$ and $B\neq 0$.
\end{lem}

\begin{prop} \label{finiter}
Suppose there is no resonance, $X_k \in \Gamma_{\frac{\delta}{N_1|r|}}(\beta_l), X_{k-r}\in \Gamma_{\frac{\delta}{|r|}}(\beta_l)$. We fix arbitrary $0<\alpha < 1$, and then fix arbitrary $0 \leq \xi < \alpha$.  There exists $\eps_0 > 0$ such that for all $0 < \eps < \eps_0$ and $|\zeta| \lesssim \eps^{\alpha -1}$,  there exists $C$ such that we have the following:
\begin{enumerate}
\item For all $r$ such that $|r| \leq M$ for any fixed $M > 0$, we have:
\begin{equation}\label{5.12}
|E_{i,j}(\eps,k,k-r)(\zeta)|\geq C|X_{k-r}(\zeta)| 
\end{equation}
\item For all $|r| \lesssim \eps^{-\xi}$, if $\Omega_{i,j}$ satisfies assumption \ref{omega},  for $j \neq N$ we have: 
\begin{equation}\label{rdependoneps}
|E_{i,j}(\eps,k,k-r)(\zeta)| \geq C\frac{|X_{k}(\zeta)|}{|r|^{2+\delta}}
\end{equation}
\end{enumerate}
\end{prop}

\begin{rem}
It is proved in \cite{W4} prop 3.22 that for $j = N$:
\begin{equation} \label{EiN}
|E_{i,N}(\eps,k,k-r)(\zeta)| \geq C\frac{|X_{k}(\zeta)|}{|r|} \text{ or } |E_{i,N}(\eps,k,k-r; \beta)|\geq C|X_{k-r}| 
\end{equation}
\end{rem}

Before the proof of proposition \eqref{finiter}, we first establish a set of notations and results from \cite{W4}:\\

We recall that $X_k = \zeta + k\frac{\beta_l}{\eps}$.  Define $\tilde{X_k}$, the  orthogonal projection of $X_k$ onto the line $\mathcal{L}(\beta_l) = \{t\beta_l: t \in \RR\}$:
\begin{equation} \label{::1}
\tilde{X}_k := (X_k\cdot \beta_l)\beta_l = s\frac{\beta_l}{\eps} + k\frac{\beta_l}{\eps}
\end{equation}  
where $s\frac{\beta_l}{\eps}$ is the orthogonal projection of $\zeta$ on $\beta_l$.
Recall that
\begin{align}
E_{i,j}(\eps,k,k-r) = \omega_i(X_k) - r\frac{\omega_N(\beta_l)}{\eps} - \omega_j(X_{k-r}).
\end{align}
We define:
\begin{align} \label{tildeEtij}
\tilde{E}_{i,j}(\eps,k,k-r) = \omega_i(\tilde{X}_k) - r\frac{\omega_N(\beta_l)}{\eps} - \omega_j(\tilde{X}_{k-r}).
\end{align}
\begin{lem} \cite{W4}
For $\tilde{X}_k$,  $\tilde{E}_{i,j}(\eps,k,k-r)$ defined above, letting $t = s+ k-r$, we have:
\begin{align} \label{Eijskr}
\tilde{E}_{i,j}(\eps,k,k-r) &= \frac{t-t_p}{\eps}(\omega_i(\beta_l)-\omega_j(\beta_l)) := \frac{t-t_p}{\eps}C(\beta_l)\\
&= \frac{s+k-r-r\Omega_{i,j}}{\eps}C(\beta_l)
\end{align}
where $
t_p(\beta_l) = r\Omega_{i,j}(\beta_l)$, $\Omega_{i,j}(\beta_l) := \frac{\omega_i(\beta_l) - \omega_N(\beta_l)}{\omega_j(\beta_l)- \omega_i(\beta_l)}
$, 
and it follows that
\begin{equation} \label{EijgEijtilde}
|E_{i,j}(\eps,k,k-r)| \geq |\tilde{E}_{i,j}(\eps,k,k-r)| - C(|X_k - \tilde{X}_k| + |X_{k-r} - \tilde{X}_{k-r}|).
\end{equation}

\end{lem}
Now we're ready to prove proposition \ref{finiter}:

\begin{proof} of proposition \ref{finiter}. \\
\textit{(1) The case where there are finitely many $r$'s ($|r| \leq M$)}.\\

Since there is no resonance, $\Omega_{i,j}$ is irrational. Picking $k_0,r_0$ such that 
\begin{equation}
C(\Omega_{i,j},M) := \left|\frac{k_0-r_0}{r_0} - \Omega_{i,j}\right| \leq \left|\frac{k-r}{r}- \Omega_{i,j}\right| \text{ for all } k \in \ZZ, |r| \leq M
\end{equation}
we have:
\begin{align} \label{452}
\frac{|k-r-r\Omega_{i,j}|}{\eps} \geq C(\Omega_{i,j},M)\eps^{-1}.
\end{align}
We rewrite equation \eqref{Eijskr} as:
\begin{equation}
\begin{aligned}
\tilde{E}_{i,j}(\eps,k,k-r)(\zeta)=
\left(\frac{s}{\eps}\right)C(\beta_l) + \left(\frac{k -r-r\Omega_{i,j}}{\eps}\right)C(\beta_l).
\end{aligned}
\end{equation}
By lemma \ref{epslemma},  inequality \eqref{452}, and $\frac{s}{\eps} \sim \zeta \lesssim \eps^{\alpha-1}$, the second term above absorbs the first, that is
\begin{align}
|\tilde{E}_{i,j}(\eps,k,k-r)| \gtrsim \frac{|k-r-r\Omega_{i,j}|}{\eps}.
\end{align}  
Using equation \eqref{EijgEijtilde}:
\begin{equation} 
\begin{aligned}
|E_{i,.j}(\eps,k,k-r)| &\geq |\tilde{E}_{i,j}(\eps,k,k-r)| - C(|X_k-\tilde{X}_k| + |X_{k-r}-\tilde{X}_{k-r}|) 
\end{aligned}
\end{equation}
and
\begin{equation} 
|\tilde{X}_{k-r} - X_{k-r}| = |\tilde{X}_k - X_k| = \left|\frac{s}{\eps} -\zeta\right| \sim \eps^{\alpha-1}
\end{equation}
we have by lemma \ref{epslemma}
\begin{align} \label{4.61}
|{E}_{i,j}(\eps,k,k-r)| \gtrsim \frac{|k-r-r\Omega_{i,j}|}{\eps}.
\end{align}
If $k = 0$, we have 
\begin{align}
|X_{k-r}|  \lesssim \frac{|k-r|}{\eps} \sim \frac{|r|}{\eps},\quad
|{E}_{i,j}(\eps,k,k-r)| \gtrsim \frac{|k-r-r\Omega_{i,j}|}{\eps} \sim \frac{|r|}{\eps}
\end{align}
which proves the proposition. If $k \neq 0$, we note that for every $r$, there is exact one $k(r) \neq 0$ such that 
\begin{align}
\left|1-\frac{r+r\Omega_{i,j}}{k(r)}\right| \leq \left|1-\frac{r+r\Omega_{i,j}}{k}\right| \quad \text{ for } k \in \ZZ \setminus 0.
\end{align}
We define 
\begin{equation}
C_0 = \inf_{|r| \leq M} \left|\frac{k(r)-r-r\Omega_{i,j}}{k(r)}\right|
\end{equation}
so 
\begin{equation}
C_0|k| \leq \left|1-\frac{r+r\Omega_{i,j}}{k}\right||k| = |k-r-r\Omega_{i,j}| \text{ for all } k.
\end{equation}
Equations \eqref{4.61} then gives
\begin{equation}
|{E}_{i,j}(\eps,k,k-r)| \gtrsim \frac{|k-r-r\Omega_{i,j}|}{\eps} \geq C_0\frac{|k|}{\eps}.
\end{equation}
Also using
\begin{align} 
|X_{k-r}| \lesssim \frac{|k-r|}{\eps} \leq \frac{|k|}{\eps} + \frac{M}{\eps} \leq (M+1)\frac{|k|}{\eps}
\end{align}
above two equations show that \eqref{5.12} holds, which finishes the proof of part (1) of the proposition. \\

\textit{(2) The case where the number of $r$'s depend on $\eps$, $|r| \leq C_0\eps^{-\xi}$,  for $0 \leq \xi < \alpha$}.\\

Since there is no resonance, $\Omega_{i,j}$ is irrational. If $\Omega_{i,j}$ further satisfies assumption \ref{omega},  \eqref{thsirothm} implies that
\begin{align} \label{diop}
\left|\frac{k-r}{r} - \Omega_{i,j}\right| \geq \frac{C(\Omega_{i,j}, \delta)}{|r|^{2+\delta}} \quad \text{ for } \delta >0.
\end{align}
Using $|r| \leq C_0\eps^{-\xi}$, we have:
\begin{align} 
\frac{|k-r-r\Omega_{i,j}|}{\eps} \geq \frac{C(\Omega,_{i,j} \delta)}{\eps|r|^{1+\delta}} \gtrsim \eps^{\xi+\xi\delta-1}.
\end{align}
We write $\tilde{E}_{i,j}$ in the same way as before: 
\begin{equation}
\begin{aligned}
\tilde{E}_{i,j}(\eps,k,k-r)(\zeta)=
\left(\frac{s}{\eps}\right)C(\beta_l) + \left(\frac{k -r-r\Omega_{i,j}}{\eps}\right)C(\beta_l).
\end{aligned}
\end{equation}
Since $0 \leq \xi < \alpha$, $\xi+\xi\delta-1 < \alpha -1$ for $\delta$ small,  by lemma \ref{epslemma}, the second term absorbs the first, that is
\begin{align} \label{4.71}
|\tilde{E}_{i,j}(\eps,k,k-r)| \gtrsim \frac{|k-r-r\Omega_{i,j}|}{\eps}.
\end{align}  
Since we have
\begin{equation} \label{hdistxk}
|\tilde{X}_{k-r} - X_{k-r}| = |\tilde{X}_k - X_k| = \left|\frac{s}{\eps} -\zeta\right| \sim \eps^{\alpha-1}
\end{equation} 
and equation  \eqref{EijgEijtilde}
\begin{equation}  \label{4.73}
\begin{aligned}
|E_{i,.j}(\eps,k,k-r)| &\geq |\tilde{E}_{i,j}(\eps,k,k-r)| - C(|X_k-\tilde{X}_k| + |X_{k-r}-\tilde{X}_{k-r}|) 
\end{aligned}
\end{equation}
we obtain using lemma \ref{epslemma} and equations \eqref{4.71},\eqref{hdistxk}, \eqref{4.73} :
\begin{align}
|{E}_{i,j}(\eps,k,k-r)| \gtrsim \frac{|k-r-r\Omega_{i,j}|}{\eps}.
\end{align}
Suppose $k = 0$, above equation becomes $|{E}_{i,j}(\eps,k,k-r)|\gtrsim \eps^{\xi + \xi \delta -1}$ by \eqref{diop}, and we have $|X_k| = |\zeta| \sim \eps^{\alpha-1}$.  Thus, lemma \ref{epslemma} implies the desired inequality.\\
Now for $k \neq 0$, we claim that there exists $C$ independent of $k,r,\eps$ such that: 
\begin{equation}
|k-r-r\Omega_{i,j}| \geq C\frac{|k|}{|r|^{2+\delta}}.
\end{equation}
Observe that by lemma \ref{epslemma} we have:
\begin{align} 
|X_{k}| &= \left|\zeta + k\frac{\beta_l}{\eps}\right| \lesssim \frac{|k|}{\eps}.
\end{align}
Therefore, the claim and above equation directly imply the result of proposition part (2).  Now It is only left to prove the claim. \\

\textit{Proof of claim.}

For each $r$ such that $|r| \leq C_0\eps^{-\xi}$. Let $k^+(r),k^-(r)$ be the integers such that 
\begin{align} \label{kclosetor}
&(a)\ k^+(r) -r - r \Omega_{i,j} > 0,\quad
k^-(r) -r - r \Omega_{i,j} < 0.\\
&(b)\ |k^+(r) - k^-(r)| = 1.\\
&(c)\ \text{One of } k^+(r), k^-(r) \text{ satisfies } |k^\pm(r)-r - r\Omega_{i,j}| = \min_{k\in \ZZ}|k-r-r\Omega_{i,j}|.
\end{align}
Since there is no resonance and  $\Omega_{i,j}$ is irrational,  $|k(r)^{\pm}-r - r\Omega_{i,j}| > 0$. The proof of the claim has two steps:

\textit{Step (1).}  We claim that for all $r$ and for $k(r) = k(r)^+$ satisfying \eqref{kclosetor} (the same argument shows the case when \textbf{$k(r) = k(r)^-$}),  there is a constant $C$ independent of $r$ such that:
\begin{equation} \label{firststep}
C\frac{|k(r)|}{|r|^{2+\delta}} \leq |k(r)-r-r\Omega_{i,j}|.
\end{equation}
Using \eqref{diop}, it is sufficient to show:
\begin{equation}
C\frac{|k(r)|}{|r|^{2+\delta}} \leq \frac{C(\Omega_{i,j},\delta)}{|r|^{1+\delta}}.
\end{equation}
If $k(r) = 0$, above equation obviously holds, otherwise it holds if:
\begin{equation}
C \leq \frac{|r|}{|k(r)|}.
\end{equation}
By the definition of $k(r)^+$, as $r \rightarrow \infty$, $\frac{|r|}{|k(r)|}\rightarrow \left|\frac{1}{1+\Omega_{i,j}}\right|$. Thus, there exists a positive constant $C$ independent of $r$ such that above equation holds, and thus \eqref{firststep} holds.\\



\textit{Step (2).} For a fixed $r = r_0,\ k_0 = k(r_0)^+$, we claim that for every $k \geq k_0^+$:
\begin{equation} \label{secondstep}
 C\frac{|k|}{|r_0|^{2+\delta}} \leq |k-r_0-r_0\Omega_{i,j}|.
\end{equation}
Let $z = k- k_0$,  we rewrite above equation as:
\begin{equation} \label{secondequation}
C\frac{|k_0+z|}{|r_0|^{2+\delta}} \leq |k_0 + z -r_0-r_0\Omega_{i,j}|.
\end{equation}
We observe that the left hand side satisfies
\begin{equation}
C\frac{|k_0+z|}{|r_0|^{2+\delta}} \leq C\frac{|k_0|}{|r_0|^{2+\delta}}+C\frac{|z|}{|r_0|^{2+\delta}} \leq |k_0-r_0-r_0\Omega_{i,j}| + C|z|.
\end{equation}
Let $x = k_0 - r_0 -r_0\Omega_{i,j}$. It is sufficient to show
\begin{equation}
|x| + C|z| \leq |x+z|.
\end{equation}
If $z = 0$ it's obviously true. Otherwise, shrink $C$ if necessary so that:
\begin{equation} \label{xzsamesign}
C \leq \frac{|x+z|-|x|}{|z|}.
\end{equation}
The definition of $k(r_0)^+$ implies $x >0$, and thus for every $z > 0$, the right hand side of above equation equals to 1 and is thus satisfied by the same $C$ established in \textit{proof of claim step (1)} (for $C<1$). Thus equation \eqref{secondstep} holds for all $k > k_0^+$.\\

Using the same argument in this step, we can show that for all $k < k_0^-$, \eqref{secondstep} holds. Since $|k_0^+ - k_0^-| = 1$, for all integers $k$, equation \eqref{secondstep} holds. This finishes the proof of the claim, and thus the proposition.


\end{proof}

\begin{rem}
Whenever we used lemma \ref{epslemma} to absorb a smaller term,  the constant and $\eps_0$ we chose only depend on the upper bound of the coefficient of the term being absorbed. Therefore since we are only absorbing terms whose coefficient is bounded above, the constant and $\eps_0$ chosen does not dependent on $r,k$.
\end{rem}

\subsection{Uniform Estimation Theorem} \label{secestimate}\hfill \\

Let $\chi(\sigma,\gamma, \eta)$ be the characteristic function on the set $|\zeta| \leq \eps^{\alpha-1}$, that is:
\begin{align} \label{chi}
\chi(\sigma,\gamma, \eta) = 
\begin{cases}
1  \quad |\zeta| \leq \eps^{\alpha-1}\\
0 \quad \text{otherwise}
\end{cases}
\end{align}
Let $\chi_D$ be the Fourier multiplier associated to $\chi(\sigma,\gamma, \eta)$, ie. $\chi_D$ is defined by:
\begin{equation} \label{chiD}
\widehat{(\chi_D U^\gamma)}_k = \chi(\sigma,\gamma, \eta) V_k = \chi(\sigma,\gamma, \eta)\widehat{U}_k
\end{equation}

Now we are ready to prove the iteration estimate for the singular system \eqref{i9} on the small/medium frequency region which is:
\begin{align}\label{i9'}
\begin{split}
&D_{x_2}V_k-\mathcal{A}(X_k)V_k=i\sum_{r\in\ZZ\setminus 0}  \alpha_r e^{ir\frac{\omega_N(\beta_l)}{\eps}x_2}B_2^{-1}MV_{k-r}+\chi\widehat{F_k}(x_2,\zeta)\\
&BV_k=\chi\widehat{G_k}(\zeta)\text{ on }x_2=0.
\end{split}
\end{align}

\begin{rem}
At this point, we don't know the existence of solutions to the system \eqref{i9'}. Therefore, the following two estimates \eqref{t30c} and \eqref{mainestimate} will be presented as a priori estimates. 
\end{rem}

\begin{rem}
We here point out again that the assumption of more than one incoming modes $\phi_{N-p},...\phi_N$ raised new difficulties in the analysis that the theorem 2.12 in \cite{W4} didn't treat. The major differences between the proof of the uniform estimate in this paper and the proof in \cite{W4} are the following,  where we take advantage of the small/medium frequency cutoff $\chi$.
\begin{enumerate}
\item Using proposition  \ref{finiter}, we "control" the amplification factor for $|r| \leq C_0\eps^{-\xi}$ in step (4) of the proof of proposition \ref{tt30}.  
\item We control the amplification factor for the  remaining $r$'s using the decay of the Fourier coefficients of $\mathcal{D}$, namely $\alpha_r$.  This is done in the proof of the main estimation theorem \ref{mainthm}.
\end{enumerate}
\end{rem}

\begin{prop}[Iteration estimate]\label{tt30}
Let $U^\gamma(t,x,\theta) \in H^2(\Omega \times \mathbb{T})$ for $\gamma \geq 1$ and set  $V_k(x_2,\zeta) = \widehat{U}_k(\zeta,x_2)$.  Assume that $\Upsilon_0 = \{\beta_l, -\beta_l\}$,  and assumptions \ref{omega}, \ref{stricthyperbolicity}, \ref{BNp}, and \ref{lopassump} are satisfied. 
There exist positive constants $C$, $\gamma_0$ such that for $\gamma\geq \gamma_0$, $V_k$ satisfies for all $k\in\ZZ$,
\begin{equation}\label{t30c}
||\chi V_k||\leq \frac{C}{\gamma}\sum_{r\in\ZZ\setminus 0}\sum_{t\in \ZZ}||\alpha_r\alpha_t\DD(\eps,k,k-r)\chi V_{k-r-t}||+\frac{C}{\gamma^2}\left|\chi\widehat{F_k}|X_k|\right|_{L^2}+\frac{C}{\gamma^{3/2}}\left|\chi \widehat{G_k}|X_k|\right|_{L^2}
\end{equation}
with $\widehat{F_k}(x_2,\zeta)$ and $\widehat{G_k}(\zeta)$ defined as:
\begin{align}
\begin{split}
\widehat{F_k}(x_2,\zeta) := &D_{x_2}V_k-\mathcal{A}(X_k)V_k - i\sum_{r\in\ZZ\setminus 0}  \alpha_r e^{ir\frac{\omega_N(\beta_l)}{\eps}x_2}B_2^{-1}MV_{k-r}\\
\widehat{G_k}(\zeta):= &BV_k(0,\zeta)
\end{split}
\end{align}
Here we redefine $\alpha_0$ to be $1$. \footnote{In \eqref{t30c} $\|V_k\|=\|V_k\|_k$ and  $\|\alpha_r\alpha_t \DD(\eps,k,k-r)V_{k-r-t}\|=\|\alpha_r\alpha_t \DD(\eps,k,k-r)V_{k-r-t}\|_{k-r-t}$.}
\end{prop}

The proof of proposition \ref{t30c} relies on a decomposition of the support of $\chi_b(\eps,k)$. For a given $r\in\ZZ\setminus 0$ we write :
\begin{align}
\chi_b(\eps,k)=\chi^1_b(\eps,k,k-r)+\chi^2_b(\eps,k,k-r)+\chi^3_b(\eps,k,k-r),
\end{align}
where $\chi^i_b(\eps,k,k-r)$, $i=1,2,3$ are respectively the characteristic functions of 
\begin{align} \label{decomp_chi}
\begin{split}
&\mathcal{A}_1(\eps,k,k-r):=\{\zeta\in\Xi: X_k\in\Gamma_{\frac{\delta}{N_1|r|}}(\beta_l), X_{k-r}\in\Gamma_{\frac{\delta}{|r|}}(\beta_l)\}\\
&\mathcal{A}_2(\eps,k,k-r):=\{\zeta\in\Xi: X_k\in\Gamma_{\frac{\delta}{N_1|r|}}(\beta_l), X_{k-r}\notin\Gamma_{\frac{\delta}{|r|}}(\beta_l)\}\\
&\mathcal{A}_3(\eps,k,k-r):=\{\zeta\in\Xi: X_k\in \Gamma_\delta(\beta_l)\setminus \Gamma_{\frac{\delta}{N_1|r|}}(\beta_l)\}.
\end{split}
\end{align}
for $N_1 \in \NN$ large and $\delta$ small. 

Now we summarize the related work done in \cite{W4} below:
\begin{lem} \cite{W4} \label{wiites} Under the same condition of proposition \eqref{tt30},  the following are true:
\begin{enumerate}
\item It is sufficient to prove \eqref{t30c} in the case where $F = 0$.
\item Equation \eqref{t30c} holds on the support of $\chi_g$, $\chi_b^2$,  and $\chi_b^3$ that is:
\begin{align} \label{chigb2b3}
||\chi_g(\eps,k )\chi V_k||&\leq \frac{C}{\gamma}\sum_{r\in\ZZ\setminus 0}\sum_{t\in \ZZ}||\alpha_r\alpha_t\DD(\eps,k,k-r)\chi V_{k-r-t}||+\frac{C}{\gamma^2}\left|\chi\widehat{F_k}|X_k|\right|_{L^2}+\frac{C}{\gamma^{3/2}}\left|\chi \widehat{G_k}|X_k|\right|_{L^2}\\
||\chi_b^2(\eps,k )\chi V_k||&\leq \frac{C}{\gamma}\sum_{r\in\ZZ\setminus 0}\sum_{t\in \ZZ}||\alpha_r\alpha_t\DD(\eps,k,k-r)\chi V_{k-r-t}||+\frac{C}{\gamma^2}\left|\chi\widehat{F_k}|X_k|\right|_{L^2}+\frac{C}{\gamma^{3/2}}\left|\chi \widehat{G_k}|X_k|\right|_{L^2}\\
||\chi_b^3(\eps,k )\chi V_k||&\leq \frac{C}{\gamma}\sum_{r\in\ZZ\setminus 0}\sum_{t\in \ZZ}||\alpha_r\alpha_t\DD(\eps,k,k-r)\chi V_{k-r-t}||+\frac{C}{\gamma^2}\left|\chi\widehat{F_k}|X_k|\right|_{L^2}+\frac{C}{\gamma^{3/2}}\left|\chi \widehat{G_k}|X_k|\right|_{L^2}
\end{align}
\item When $F = 0$, we have
\begin{align}\label{t33}
\begin{split}
&(a) |\chi_b(\eps,k)\tilde w^+_k|_{L^2}\leq \frac{C}{\gamma}\sum_{r\in\ZZ\setminus 0}\left(\left|\frac{\chi_b(\eps,k)\ar w^+_{k-r}}{\Delta(\eps,k)}\right|_{L^2}+\left|\frac{\chi_b(\eps,k)\ar w^-_{k-r}}{\Delta(\eps,k)}\right|_{L^2}\right),\\
&(b) |\chi_b(\eps,k)w^-_k|_{L^2}\leq \frac{C}{\gamma}\sum_{r\in\ZZ\setminus 0}\left(\left|\frac{\chi_b(\eps,k)\ar w^+_{k-r}}{\Delta(\eps,k)}\right|_{L^2}+\left|\frac{\chi_b(\eps,k)\ar w^-_{k-r}}{\Delta(\eps,k)}\right|_{L^2}\right)+\frac{C}{\gamma^{3/2}}\left|\widehat{G_k}|X_k|\right|_{L^2(\sigma,\eta)}.
\end{split}
 \end{align}
\item When $F = 0$, we can replace (not dominate) $\frac{C}{\gamma}\left|\frac{\chi^1_b(\eps,k,k-r)\ar w^+_{k-r}}{\Delta(\eps,k)}\right|_{L^2} $ the term on the right of \eqref{t33} by:
\begin{align} \label{step4}
\frac{C}{\gamma}\sum_{t\in\ZZ}\left|\ar\at D(\eps,k,k-r;\beta_l)\CalW_{k-r-t}\right|_{L^2}.
\end{align}
\end{enumerate}
\end{lem}

\begin{proof} \label{newamp}  of proposition \ref{tt30}.\\

We observe that lemma \label{wiites}(2) implies that it's sufficient to prove \eqref{t30c} on the support of $\chi_b^1$.  \\

\textbf{Claim}: We can also replace $\frac{C}{\gamma}\left|\frac{\chi^1_b(\eps,k,k-r)\ar w^-_{k-r}}{\Delta(\eps,k)}\right|_{L^2} $ on the right of \eqref{t33} by \eqref{step4}.  \\

Suppose the claim is true, using proposition \ref{g1z} and the formula \eqref{a5} \eqref{a6}, we see that $\frac{1}{\sqrt{\gamma}}|\tilde{w}_k^+(0,\zeta)|_{L_2(\sigma,\eta)},\frac{1}{\sqrt{\gamma}}|w_k^-(0,\zeta)|_{L_2(\sigma,\eta)}$ satisfy, respectively, the same estimates as $\frac{1}{\sqrt{\gamma}}|\tilde{w}_k^+(x_2,\zeta)|_{L_2},\frac{1}{\sqrt{\gamma}}|w_k^-(x_2,\zeta)|_{L_2}$.  Therefore, \eqref{step4} together with \eqref{t33} gives the estimate \eqref{t30c} on the support of $\chi_b^1$, which finishes the proof of proposition \ref{tt30}. It is only left to prove \eqref{step4}.

\textit{Proof of claim: improving the $\left|\frac{\chi^1_b(\eps,k)\ar w^-_{k-r}}{\Delta(\eps,k)}\right|_{L^2}$ pieces of \eqref{t33}.} 

For any fixed $r$, the main part of $\left|\frac{\chi^1_b(\eps,k)\ar w^-_{k-r}}{\Delta(\eps,k)}\right|_{L^2}$ arises from the part of \eqref{a6} given by:
 \begin{align}\label{h14}
 \begin{split}
 &A:=e^{i\xi_-(\eps,k)x_2} [Br_-(\eps,k)]^{-1}Br_+(\eps,k)\int^\infty_0 e^{i\xi_+(\eps,k)(-s)+i\frac{r\omega_N(\beta_l)}{\eps}s} \ar b(\eps,k,k-r)w^-_{k-r}(s,\zeta)]ds.
\end{split}
\end{align}

We ignore factors that are $O(1)$ and independent of $s$,  and the $p-$component of $A$ is a sum of terms of the form
 \begin{align}\label{h15}
A_{p,i,j}:=e^{i\omega_p(\eps,k)x_2}\frac{1}{\Delta(\eps,k)}\int^\infty_0 e^{i\omega_i(\eps,k)(-s)+i\frac{r\omega_N(\beta_l)}{\eps}s} \alpha_r w^-_{k-r,j}(s,\zeta)ds
\end{align}
where $p\in \mathcal{I}$, $i\in\mathcal{O}$ and $w^-_{k-r,j}$, $j\in\mathcal{I}$ denotes a component of $w^-_{k-r}$. We now improve the $\chi_b^1 A_{p,i,j}$ by an integration by part. Setting:
\begin{equation}
w^{*,-}_{k-r,j}(x_2,\zeta):= e^{-i\omega_j(\eps,k-r)x_2}w_{k-r,j}^-
\end{equation}
we may rewrite $\chi_b^1 A_{p,i,j}$ as:
\begin{equation} \label{482}
e^{i\omega_p(\eps,k)x_2}\frac{1}{\Delta(\eps,k)}\int^\infty_0 e^{-iE_{i,j}(\eps,k,k-r)s}\chi_b^1\alpha_r w^{*,-}_{k-r}(x_2,\zeta)ds
\end{equation}
where we recall that $E_{i,j}(\eps,k,k-r):= \omega_i(\eps,k) - \frac{r\omega_N(\beta_l)}{\eps}-\omega_j(\eps,k-r)$. From \eqref{a3a}, we have:
\begin{equation} \label{483}
\D_{x_2}w^{*,-}_{k-r,j} = e^{-i\omega_j(\eps,k-r)x_2}ih_{k-r,j}
\end{equation}
where $h_{k-r,j}$ is the $j$th component of the right hand side of $\eqref{a3a}_{k-r}$. We also have:
\begin{equation}\label{484}
-\frac{1}{i} E_{i,j}^{-1}\frac{d}{ds}e^{-iE_{i,j}s} = e^{-i E_{i,j}s}
\end{equation}
Using \eqref{483} \eqref{484}, an integration by part on \eqref{482} gives:
\begin{align} \label{ibpresult}
e^{i\omega_p x_2}&\frac{1}{i\Delta (\eps,k)}\chi_b^1 E_{i,j}^{-1}(\eps,k,k-r) \left[\alpha_r w_{k-r,j}^-(0,\zeta) + \int_0^\infty e^{-iE_{i,j}(\eps,k,k-r)s} e^{-i\omega_j(\eps,k-r)s}\ar i h_{k-r}(s,\zeta)\right]ds
\end{align}
Now by proposition \ref{finiter}, in particular \eqref{rdependoneps},  on the support of $\chi_b^1$,  for $|r|\leq \eps^{-\xi}$, when $j \neq N$ we have:
\begin{equation} \label{control}
\left| \frac{1}{\Delta(\eps,k)E_{1,j}}\right|\leq \frac{|X_k|}{\gamma}\frac{|r|^{2+\delta}}{C|X_k|} \sim \frac{|r|^{2+\delta}}{\gamma}
\end{equation}
and for $j = N$, using \eqref{EiN},  we have the following three cases:
(I). The first alternative in \eqref{EiN} holds; (II). The second alternative holds and $|X_{k-r}| \geq \frac{|X_k|}{N_2}$; (III). The second alternative holds and $|X_{k-r}| < \frac{|X_k|}{N_2}$  for $N_2 \in \ZZ$ chosen below. \\

We have, in case (I):
\begin{equation} \label{controlN1}
\left| \frac{1}{\Delta(\eps,k)E_{i,N}}\right|\leq \frac{|X_k|}{\gamma}\frac{|r|}{C|X_k|} \sim \frac{|r|}{\gamma},
\end{equation}

in case (II):
\begin{equation} \label{controlN2}
\left| \frac{1}{\Delta(\eps,k)E_{i,N}}\right|\leq \frac{|X_k|}{\gamma}\frac{1}{C|X_{k-r}|} \sim \frac{N_2}{\gamma}
\end{equation}

in case (III): the third term in $E_{i,j}$ is dominated by the first two terms,  more precisely,
\begin{equation}\label{h4'}
|X_{k-r}| = |X_k - \frac{r\beta_l}{\eps}| \leq \frac{1}{N_2}|X_k|,
\end{equation}
so taking $N_2$ large, we have\footnote{Use \eqref{h4'} and \eqref{t6a}(e) for the last $\sim$ of \eqref{h6'}.}:
\begin{equation} \label{h6'}
|\omega_N(X_{k-r})| \lesssim |X_{k-r}| \leq \frac{1}{N_2}|X_k| \sim |\omega_i(X_k) - \omega_N(\frac{r\beta_l}{\eps})|
\end{equation}
This implies that $E_{i,N}$ satisfies 
\begin{equation} \label{controlN3}
|E_{i,N}^{-1}| \leq \frac{C(N_2)}{|X_k|} \text{ and hence } \left|\frac{1}{\Delta(\eps,k)E_{i,N}} \right| \leq \frac{C(N_2)}{\gamma}
\end{equation}
provided $N_2$ is large enough.

Equation \eqref{ibpresult}, \eqref{control},  \eqref{controlN1}, \eqref{controlN2}, \eqref{controlN3}, proposition \ref{g1z}, and definition \eqref{t28} yield the desired inequality on the support of $\chi$:
\begin{align} \label{384}
&|\chi_b^1(\eps,k,k-r) \chi A_{p,i,j}|_{L^2} \lesssim 
\\&\quad \frac{C}{\gamma^{3/2}}\left[|\ar D(\eps,k,k-r)\chi w_{k-r,j}^{-}(0,\zeta)|_{L^2(\sigma,\eta)} + \frac{C}{\sqrt{\gamma}}\sum_{t\in\ZZ \setminus 0}|\ar\alpha_t D(\eps,k,k-r)\chi w_{k-r-t}|_{L^2}\right]
\end{align}

We observe that when $|r|\geq \eps^{-\xi}$,  above inequality holds by the definition of amplification factor.

The piece of the second term on the right of \eqref{t33}(a) given by $\left|\frac{\chi^1_b(\eps,k)\ar w^-_{k-r}}{\Delta(\eps,k)}\right|_{L^2}$ arises from the following part of \eqref{a5}:
\begin{equation}
\frac{\chi_b^1(\eps,k,k-r)}{\Delta(\eps,k)}\int_{x_2}^\infty e^{i\xi_+(\eps,k)(x_2-s)+i\frac{\omega_N(\beta_l)}{\eps}s}\ar b(\eps,k,k-r)w^-_{k-r}(s,\zeta)ds
\end{equation}
A same integration by part argument would give the estimate of the form \eqref{384}.\\

Now as we finished the proof of claim,  the proof of proposition \ref{tt30} is completed.
\end{proof}

The next theorem gives an a priori estimate in the small/medium frequency region for the singular system:
\begin{align}\label{i6s}
\begin{split}
&D_{x_2}U+A_0(D_t+\frac{\sigma_l}{\eps}D_\theta)U+A_1(D_{x_1}+\frac{\eta_l}{\eps}D_\theta)U-iB_2^{-1}\cD\left(\frac{\omega_N(\beta_l)}{\eps}x_2+\theta\right)U= F(t,x,\theta)\\
&BU= G(t,x_1,\theta)\text{ on }x_2=0\\
\end{split}
\end{align}

\begin{thm} \label{mainthm}
Let $U^\gamma(t,x,\theta) \in H^2(\Omega \times \mathbb{T})$ for $\gamma \geq 1$.    Assume that the singular system \eqref{i6s} has $p$ incoming phases $\phi_{N-p},...,\phi_N$.  Let:
\begin{align}\label{i6s'}
\begin{split}
& F(t,x,\theta) := D_{x_2}U+A_0(D_t+\frac{\sigma_l}{\eps}D_\theta)U+A_1(D_{x_1}+\frac{\eta_l}{\eps}D_\theta)U-iB_2^{-1}\cD\left(\frac{\omega_N(\beta_l)}{\eps}x_2+\theta\right)U\\
& G(t,x_1,\theta) := BU(t,x_2,0,\theta)\\
\end{split}
\end{align}
Suppose the coefficients $\widehat{\mathcal{D}}(r)$ satisfy  $\widehat{\mathcal{D}}(0) = 0$ and $\widehat{\mathcal{D}}(r) \lesssim |r|^{-(M+N)}$ for some $M \geq 4$ and $N \geq \xi^{-1}$ where $\xi$ is as in proposition \ref{finiter}. Under the usual structural  assumptions \ref{stricthyperbolicity}, \ref{lopassump}, and \ref{BNp},  we assume further that $\Upsilon_0=\{\beta_l,-\beta_l\}$ and assumption \ref{omega} holds. Then there exist positive constants $\eps_0,\ \gamma_0,\ K$ such that for $0 < \eps \leq \eps_0$ and $\gamma \geq \gamma_0$ we have:
\begin{align} \label{mainestimate}
|\chi_D U^\gamma|_{L^2(t,x,\theta)} + \left|\frac{\chi_D U^\gamma (0)}{\sqrt{\gamma}}\right|_{L^2(t,x_1,\theta)} \leq 
K\left[ \frac{1}{\gamma^2}(\sum_{k\in\ZZ}\left| \chi |X_k|\widehat{F_k}\right|^2_{L^2(x_2,\sigma,\eta)})^{1/2} + \frac{1}{\gamma^{3/2}}(\sum_{k\in\ZZ}\left| \chi|X_k|\widehat{G_k}\right|^2_{L^2(\sigma,\eta)})^{1/2}\right]
\end{align}
\end{thm}
\begin{proof}
We note that by equation \eqref{tu31}, the proof is reduced to the proof of the following equation:
\begin{equation} \label{l2vk}
(||\chi V_k||)_{\ell^2(k)} \leq 
K\left[ \frac{1}{\gamma^2}(\sum_{k\in\ZZ}\left| \chi |X_k|\widehat{F_k}\right|^2_{L^2})^{1/2} + \frac{1}{\gamma^{3/2}}(\sum_{k\in\ZZ}\left| \chi|X_k|\widehat{G_k}\right|^2_{L^2(\sigma,\eta)})^{1/2}\right]
\end{equation}

Define \begin{equation}
\beta_r:=
\begin{cases}
\alpha_r C|r|^{2+\delta} \text{ for } r\leq \eps^{-\xi}\\
\frac{\alpha_r C|r|}{\eps} \text{ otherwise } 
\end{cases}
\end{equation}
By definition of $\DD(\eps,k,k-r)(\zeta)$, we see that $\beta_r \geq \alpha_r\DD$. Using the fact that when $|r| \gtrsim \eps^{-\xi}$, the corresponding Fourier coefficients of $\mathcal{D}$, namely $\alpha_r$, satisfy  $|\alpha_r| \leq C_{M,N}<r>^{-(M+N)}\leq C_{M,N}\eps^{N\xi}|r|^{-M}$, we have:
\begin{align}
&|\beta_r| \leq |r|^{-3/2} \text{ for } r\leq \eps^{-\xi}\\
&|\beta_r| \leq C\eps^{N\xi-1}|r|^{-(M-1)} \leq C|r|^{1-M}\text{ otherwise } 
\end{align}
The above inequalities ensure that $|(\beta_r)|_{\ell^1}$ is finite.  From this point we can repeat the proof of Proposition 4.7 in \cite {W4} to obtain the desired estimate.  That is\footnote{We note that the definition of $\mathbb{D}$ and $\beta_r$ are different from \cite{W4}.}:
\begin{align}
\begin{split}
&\sum_{r\in\ZZ\setminus 0}\sum_{t\in\ZZ}\|\alpha_r\alpha_t\DD(\eps,k,k-r)\chi V_{k-r-t}\|\leq \sum_{r\in\ZZ\setminus 0}\sum_{t\in\ZZ}\|\beta_r\alpha_t \chi V_{k-r-t}\|=\\
&\qquad \qquad \sum_s\left(\sum_{r+t=s}|\beta_r\alpha_t|\right)\| \chi V_{k-s}\|:= \sum_s\gamma_s\|\chi V_{k-s}\|.
\end{split}
\end{align}
Young's convolution inequality  gives
\begin{align} 
 \left|\left(\sum_s\gamma_s\|\chi V_{k-s}\|\right)\right|_{\ell^2(k)}\leq |(\|\chi V_k\|)|_{\ell^2}|(\gamma_s)|_{\ell^1}. 
 \end{align}
Since $\gamma_s=\sum_r|\beta_r||\alpha_{s-r}|$,  using Young's convolution inequality again we have
\begin{align}
|(\gamma_s)|_{\ell^1}\leq |(\beta_r)|_{\ell^1}|(\alpha_t)|_{\ell^1}:=K_1.
\end{align}
Thus, by propsition \ref{tt30} with remark \ref{newamp}, we have
\begin{equation}
||\chi V_k||_{\ell^2} \lesssim \frac{K_1}{\gamma} |(\| \chi V_k\|)|_{\ell^2}+|(|\chi |X_k|\widehat{F_k}|_{L^2(x_2,\sigma,\eta)})|_{\ell^2}+|(|\chi |X_k|\widehat{G_k}|_{L^2(\sigma,\eta)})|_{\ell^2}
\end{equation}
and the result follows by taking $\gamma_0$ large enough.
\end{proof}

\subsection{Application of the Main Estimate: an Existence Theorem} \label{ext}\hfill \\

In this section, we informally discuss an application of theorem \ref{mainthm} of getting a ``restricted" existence result for a forward problem different from \eqref{i6}.  The idea is to use a classical duality argument based on an energy estimate for the corresponding ``backward" or adjoint boundary value problem together with the Riesz representation theorem. 

We first explain why the argument fails to yield an existence result for \eqref{i6}.  In \eqref{i6}, we recall that $\omega_N$ is incoming,  which is used to prove the main energy estimate \eqref{mainestimate}. If we want to prove an existence result for \eqref{i6} by duality, we need an estimate for the corresponding backward problem. In that backward problem, $\omega_N$ still appears in the oscillatory term but now it is outgoing, and thus the proof of the energy estimate breaks down for this backward problem.  In the proof of proposition \ref{tt30} (and lemma \ref{wiites}),  we generally rely on a lower bound for terms like $E_{i,j}$  to control the term $\Delta^{-1}(\eps,k)$.  More precisely,  to prove the second equation in \eqref{chigb2b3}  of lemma \ref{wiites}, we need a lower bound for $E_{i}$:
\begin{equation}
E_{i}(\eps,k,k-r)=\left(\omega_i(X_k)-r\omega_N(\frac{\beta_l}{\eps})\right)I_N-\cA(X_{k-r}).
\end{equation}
Now since $\omega_i$ is outgoing,  we may have $i = N$. In this case,  for any fixed $k = r \neq 0$, we have:
\begin{equation}
E_{i}(\eps,k,k-r)=\left(\omega_N(\zeta+r\frac{\beta_l}{\eps})-\omega_N(\frac{r\beta_l}{\eps})\right)I_N-\cA(\zeta).
\end{equation}
which can be small, for example, for $\zeta$ near $0$.  It is impossible to establish an lower bound for $E_{i}$, so we can't prove even a restricted existence result for \eqref{i6}. 

Given above observation, in this section, we are considering the following singular problem that is different from \eqref{i6} for which $\omega_N$ is outgoing. This means that the new corresponding backward problem $\omega_N$ is incoming, so we will have a restricted estimate for this backward problem.  We study the forward singular problem:

\begin{align}\label{u1}
\begin{split}
&D_{x_2}U+A_0(D_t+\frac{\sigma_l}{\eps}D_\theta)U+A_1(D_{x_1}+\frac{\eta_l}{\eps}D_\theta)U-iB_2^{-1}\cD\left(\frac{\omega_N(\beta_l)}{\eps}x_2+\theta\right)U=F(t,x,\theta)\\
&BU=G(t,x_1,\theta)\text{ on }x_2=0\\
&U=0\text{ in }t<0.
\end{split}
\end{align}
where we now assume that the $B_j$ used to define the $A_j$ are now symmetric (recall that  $A_0 = B_2^{-1}$,  $A_1 = B_2^{-1}B_1$). 

Suppose that the $\omega_j$  are incoming for $j\in \{1,...,N-p\}$  and outgoing for $j\in \{N-p+1,...,N\}$.  We denote the operator on the left hand side by  $L := D_{x_2}+A_0(D_t+\frac{\sigma_l}{\eps}D_\theta)+A_1(D_{x_1}+\frac{\eta_l}{\eps}D_\theta)-iB_2^{-1}\cD\left(\frac{\omega_N(\beta_l)}{\eps}x_2+\theta\right)$. 
Commuting \eqref{u1}  with $e^{-\gamma t}$, we are led to find solutions in $L^2(t,x,\theta)$ of the system:
\begin{align} \label{u2}
&L^\gamma \tilde{U} = F^\gamma\\
&B\tilde{U} = G^\gamma
\end{align}
where $L^\gamma := -i\gamma A_0 + L$,  $F^\gamma = e^{-\gamma t} F$, $G^\gamma =e^{-\gamma t}G$. We define the formal adjoint $L^*$ of $L$:
\begin{align}
L^* := &D_{x_2} + A_0^T(D_t + \frac{\sigma_l}{\eps})D_\theta + A_1^T(D_{x_1}+\frac{\eta_l}{\eps})D_\theta + \left( iB_2^{-1}\cD\left(\frac{\omega_N(\beta_l)}{\eps}x_2+\theta\right) \right)^T 
\\ (L^\gamma)^* := &i\gamma A_0^T + L^*
\end{align}
with respect to the $L^2$ inner product: 
\begin{align}
(U,V)_{L^2(\Omega \times \TT)} = \int_{\Omega \times \TT} U(t,x,\theta)\overline{V(t,x,\theta)}
\end{align}
Then, the dual problem (or backward problem) of \eqref{u2} is: 
\begin{align} \label{bwp}
&(L^\gamma)^* V = f_\sharp 
\\&M_\sharp V|_{x_2 = 0} = g_\sharp
\end{align}
for suitable $f_\sharp , g_\sharp$, and $M_\sharp$ as a $(n-p)  \times n$ matrix of maximal rank satisfying:
\begin{equation} \label{msharp}
B^T_\sharp B + M^T_\sharp M = I
\end{equation}
for suitable matrices $B_\sharp$ and $M$.

We observe that the backward problem \eqref{bwp} is of the same form as \eqref{i6s}.  In particular,  in \eqref{bwp} $\omega_j$ is outgoing for $j\in \{1,...,N-p\}$  and is incoming for $j\in \{N-p+1,...,N\}$\footnote{Here we use the assumption that the $B_j$'s are symmetric.}.  We can apply theorem \ref{mainthm} to \eqref{bwp} to get an estimate of the form:
\begin{equation} \label{es0}
\| \chi_D V\|_0 + \frac{1}{\sqrt{\gamma}} \vert \chi_D V (0)\vert _0 \leq 
C \left[ \frac{1}{\gamma^2}\|f_1\|_1 + \frac{1}{\gamma^{3/2}}\vert g_1 \vert_1\right]
\end{equation}
for appropriate $f_1, g_1$. Here $V(0):= V|_{x_2=0}$, $\|V \|_0$ denotes the $L^2(\Omega \times \TT)$ norm,  and $\vert V \vert_s$  denotes the $L^2(\RR^2 \times \TT^3)$ norm, and $\|f\|_1, |g|_1$ are the norms equivalent to the norm appeared on the right hand side of \eqref{mainestimate}.  We can obtain an estimate of the following form by commuting a singular operator with symbol $\langle X_k\rangle^{-1}$  through the problem:
\begin{equation} \label{es-1}
\| \chi_D V\|_{-1} + \frac{1}{\sqrt{\gamma}} \vert \chi_D V(0)\vert _{-1} \leq 
C \left[ \frac{1}{\gamma^2}\|f_2\|_0 + \frac{1}{\gamma^{3/2}}\vert g_2 \vert_0\right]
\end{equation}
for appropriate $f_2, g_2$.  \\

With the help of above estimates, we may apply an argument similar to that  in \cite{Cou05} to get the existence of solutions to \eqref{u1}.  We sketch this argument below.
Define a space of test functions:
\begin{equation}
Q : = \{V \in C_0^\infty(\overline{\Omega \times \TT}) \text{ s.t. } M_\sharp V = 0\}
\end{equation}
We may define a linear functional $\ell$ on $(L^\gamma)^* Q$ by:
\begin{equation} \label{ellop}
\ell[(L^\gamma)^* V] := (F^\gamma,\chi_D V)_{L^2(\Omega \times \TT)} + (G^\gamma,B_\sharp \chi_D V(0))_{L^2(\RR^2 \times \TT)}, \quad \forall V\in Q
\end{equation}
and we have estimate:
\begin{align}
|\ell[(L^\gamma)^* V]| &\leq \|F^\gamma \|_1 \|\chi_D V\|_{-1} + C|G^\gamma|_1 |\chi_D V(0)|_{-1}
\\&\leq C \left(\|\chi_D V\|_{-1} + |\chi_D V(0)|_{-1} \right)
\\&\leq \frac{C}{\gamma^{3/2}} \|(L^\gamma)^* V\|_0
\end{align}
which shows that $\ell$ is continuous. Now by Hahn-Banach theorem, we can extend $\ell$ from the subspace $(L^\gamma)^* Q$ to a continuous linear functional on the whole space of $L^2(\Omega \times \TT)$. Riesz representation theorem gives a function $\tilde{U} \in L^2(\Omega \times \TT)$ such that:
\begin{equation}
\ell[(L^\gamma)^* V] = (U,(L^\gamma)^* V)_{L^2(\Omega \times \TT)}.
\end{equation}
In particular, we have $L^\gamma \tilde{U} = \chi_D F^\gamma$ in the sense of distributions.

As for the boundary,  the trace $ \tilde{U}$ belongs to $H^{-1/2}(\RR^2 \times \TT)$, and the following Green's formula holds:
\begin{equation} \label{greens1}
(\tilde{U},(L^\gamma)^* V)_{L^2(\Omega)} = (\chi_D F^\gamma,V)_{L^2(\Omega)} + (\tilde{U}(0),V(0))_{H^{-1/2}(\RR^2 \times \TT),H^{1/2}(\RR^2 \times \TT)}, \quad \forall V \in C_0^\infty(\overline{\Omega \times \TT}).
\end{equation}
We observe that \eqref{msharp} gives:
\begin{align}\label{bdbsharp}
(\tilde{U}(0),V(0))_{H^{-1/2}(\RR^2 \times \TT),H^{1/2}(\RR^2 \times \TT)} &= (B^T_\sharp B + M^T_\sharp M \tilde{U}(0),V(0))_{H^{-1/2}(\RR^2 \times \TT),H^{1/2}(\RR^2 \times \TT)}
\\&= (B\tilde{U}(0),B_\sharp V(0)) + (M\tilde{U}(0),M_\sharp V(0))
\end{align}
Combining \eqref{ellop}, \eqref{greens1}, and \eqref{bdbsharp} gives:
\begin{equation}
(\chi_D G^\gamma - B\tilde{U}(0), B_\sharp V(0))_{H^{-1/2}(\RR^2 \times \TT^3),H^{1/2}(\RR^2 \times \TT)} = 0
\end{equation}
Using \eqref{msharp}, we observe that:
\begin{equation}
\begin{pmatrix}
B_\sharp \\ M_\sharp
\end{pmatrix} \in W^{2,\infty}(\RR^2 \times \TT)
\end{equation}
is invertible. We therefore conclude that $B\tilde{U}(0)= \chi_D G^\gamma$, which completed showing that $\tilde{U}$ is a solution to \eqref{u1} on $\chi_D$. \\

We conclude that now we have shown the existence of solutions $U \in e^{\gamma t}L^2(\Omega \times \TT_\theta)$ to the system:
\begin{align}
&D_{x_2}U+A_0(D_t+\frac{\sigma_l}{\eps}D_\theta)U+A_1(D_{x_1}+\frac{\eta_l}{\eps}D_\theta)U-iB_2^{-1}\cD\left(\frac{\omega_N(\beta_l)}{\eps}x_2+\theta\right)U=\chi_D F(t,x,\theta)\\
&BU=\chi_D G(t,x_1,\theta)\text{ on }x_2=0\\
\end{align}

\newpage
\section{Geometric Optics Solutions for Weakly Stable Hyperbolic System} \label{newresult}
In this section, we construct an approximate solution to system \eqref{eqn:msystem}. Although we use the tools established in section \ref{tfcas}, there are some new difficulties in the construction due to the failure of the uniform Lopatinski condition at $\beta_l$.\\

We study the following $3 \times 3$ system on $\Omega = (-\infty, \infty) \times \{(x_1,x_2):x_2 \geq 0\}$: 
\begin{equation} \label{eqn:msystem}
	\begin{aligned}
&L(\partial)u + \left(e^{-i\frac{\phi_3}{\epsilon}}+e^{i\frac{\phi_3}{\epsilon}}\right)Mu = 0 \text{ in } x_2>0 \\ 
&Bu = \epsilon G \left( t,x_1,\frac{\phi_0}{\epsilon} \right) \text{ on }x_2 = 0 \\ 
&u = 0 \text{ in } t<0 
	\end{aligned}
\end{equation}
where $L(\D)u := \partial_tu + B_1\partial_{x_1}u + B_2\partial_{x_2}u$, and $G$ is a periodic function in $\theta_0$.
Here all $B_j$'s are $3 \times 3$ matrices and $B_2$ is invertible. We take $\phi_1$ to be the outgoing phase, $\phi_2$, $\phi_3$ to be the incoming phase, and there is no resonance. The boundary phase is taken to be $\phi_0(t,x_1) = \beta_l \cdot (t,x_1)$, where $\beta_l =(\sigma_l, \eta_l) \in \RR^2 \setminus {0}$  is the only direction where the uniform Lopatinski condition fails. Also,  
\begin{equation}
	\begin{aligned}
	&\phi_m (t,x_1,x_2) = \beta_l \cdot (t, x_1) + \omega_m(\beta_l)x_2 \\
	&d\phi_m = (\beta_l, \omega_m)
	\end{aligned}
\end{equation}

\begin{thm}\label{apsothm}
Consider the system \eqref{eqn:msystem} where we assume $Mr_3 = 0$.\footnote{The assumption $Mr_3=0$ is used to decouple some of the profile equations below.} If the small divisor assumption \ref{sdassump} is satisfied and  there is no resonance, there exists an approximate solution  $u^\epsilon_a(t,x) = \sum_{k={1}}^J \epsilon^kU_k(t,x,\frac{\Phi}{\epsilon})$ where $U_k \in \mathcal{H}^\infty(t,x,\theta)$, periodic in $\theta$.
\begin{equation}
\begin{split}
D_{x_2}U^\eps_a+&A_0(D_t+\frac{\sigma_l}{\eps}D_{\theta_0})U^\eps_a+A_1(D_{x_1}+\frac{\eta_l}{\eps}D_{\theta_0})U^\eps_a 
\\-
i&\left(e^{i\left(\frac{\omega_3(\beta_l)}{\eps}x_2+\theta_0\right)}+ e^{-i\left(\frac{\omega_3(\beta_l)}{\eps}x_2+\theta_0\right)}\right)B_2^{-1}MU^\eps_a 
\\ &=\eps^J\left(\mathcal{L}(\D_{\theta})U_J+L(\D)U_{J-1}+\left(e^{i\left(\frac{\omega_3(\beta_l)}{\eps}x_2+\theta_0\right)}+ e^{-i\left(\frac{\omega_3(\beta_l)}{\eps}x_2+\theta_0\right)}\right)MU_{J-1}\right)
 \\
BU^\eps_a &=G(t,x_1,\theta_0)\text{ on }x_2=0
\\U^\eps_a &= 0 \text{ for }t < 0
\end{split}
\end{equation}
\end{thm}
As proof of the theorem, we construct approximate solutions.

\subsection{Profile Equations} \hfill \\

We use the same set of tools defined in section \ref{tfcas}, and apply similar method as before.  However, since $B$ fails to be an isomorphism between $\EE^s(\zeta)$ and $\CC^p$, a new discussion on boundary equation is needed, which is discussed in section \ref{abae}.  First we construct approximate solutions of the form
\begin{equation}
u_a^\epsilon (t,x) = \sum_{k=0}^J \epsilon^kU_k\left(t,x,\frac{\Phi}{\epsilon}\right)
\end{equation}
Here all $U_k \in \mathcal{H}^\infty$. Let $f(\theta_3) := e^{-i\theta_3}+e^{i\theta_3}$.  Using the same notation as before, we plug in the approximate solution to our system of equations and get:
\begin{equation} \label{eqn:interior}
	\begin{aligned}
	&\mathcal{L}(\D_\theta)U_{0} = 0 \\
	&\mathcal{L}(\D_\theta)U_1 + L(\D)U_{0} + f(\theta_3)MU_{0} = 0 \\
	&\mathcal{L}(\D_\theta)U_2 + L(\D)U_{1} + f(\theta_3)MU_{1} = 0 \\
	&\mathcal{L}(\D_\theta)U_j + L(\D)U_{j-1} + f(\theta_3)MU_{j-1} = 0 \text{ for } j \geq 1\\
	\end{aligned}
\end{equation}
and the boundary equations at $x_2 = 0, \theta_1 = \theta_2 = \theta_3 = \theta_0$:

\begin{equation}
	\begin{aligned} \label{eqn:bdy1}
	&BU_{0} = 0 \\
	&BU_{1} = G(t,x_1,\theta_0) \\
	&BU_j = 0 \text{ for } j\geq 2
	\end{aligned}
\end{equation}
Applying $E_Q$ to \eqref{eqn:interior}, we have:
\begin{equation}\label{eqn:eqeqn}
	\begin{aligned}
	&E_Q[L(\D)U_{0}+ f(\theta_3)MU_{0}] = 0\\
	&E_Q[L(\D)U_{1}+ f(\theta_3)MU_{1}] = 0\\
	&E_Q[L(\D)U_{j}+ f(\theta_3)MU_{j}] = 0 \text{ for } j \geq 0\\
	\end{aligned}
\end{equation}
And applying $R$ to \eqref{eqn:interior}, by \eqref{eqn:ROp} we have:
\begin{equation} \label{eqn:epeqn}
	\begin{aligned}
	&(I-E_P)U_{0} = 0\\
	&(I-E_P)U_{1} = -R[L(\D)U_{0}+f(\theta_3)MU_{0}]\\
	&(I-E_P)U_{2} = -R[L(\D)U_{1}+f(\theta_3)MU_{1}]\\
	&(I-E_P)U_{j} = -R[L(\D)U_{j-1}+f(\theta_3)MU_{j-1}]\text{ for } j \geq 1\\
	\end{aligned}
\end{equation}
We now proceed to define  $e,\check{e},\check{E}(\beta_l)$ as in  \cite{W4}.  By proposition \ref{esrj}, the vector space $E^s(\beta_l)$ is spanned by $\{r_2,r_3\}$. $B$ fails to be an isomorphism as in assumption \ref{lopassump} implies that the subspace $\ker B \cap E^s(\beta_l)$ is one-dimensional and is thus we can define $e = e_2 + e_3$ such that $\ker B \cap E^s(\beta_l) = \text{span} \{e\}$ for $e_m \in $ span$\{r_m\}$.  We know that the vector space $BE^s(\beta_l)$ is one-dimensional and of real type,  so we have:
\begin{equation} \label{eqn:bdot1}
BE^s(\beta_l) = \{X \in \CC^2: b \cdot X = 0\}
\end{equation}
for an appropriate nonzero vector $b$.  We next choose $\check{E}(\beta_l)$,  such that $E^s(\beta_l) = \check{E}(\beta_l) \oplus \text{span} \{e\}$.  We pick $\check{e}$ such that $\check{E}(\beta_l) = \text{span} \{\check{e}\}$, where $\check{e} = \check{e}_2 + \check{e}_3$, $\check{e}_m \in \text{span} \;r_m$. Thus, by definition we have an isomorphism:
\begin{equation}
B:\check{E}^s(\beta_l) \rightarrow BE^s(\beta_l)
\end{equation}
To analyze an equation of the form $BU_k = H(t,x,\theta_0)$, using notation \ref{eqn:Unotn}, we take the mean on both sides: $B\underline{U_k} = \underline{H}$, and the mean zero oscillation: 
\begin{equation}
(BU_k)^*=H^* \Rightarrow  (B{E_P}{U_k} + B(I-E_P)U_k)^*= H^* \Rightarrow B{E_{P_{in}}}{U_k} = H^* - E_{P_1}U_k - B[I-E_P)U_k]^*.
\end{equation} Since $B{E_{P_{in}}}U_k = B\check{U}_k$, we have:
\begin{equation}
\begin{aligned}
&B\underline{U_k} = \underline{H}\\
&B\check{U}_k = H^* - BEP_1U_k - B[(I-E_P)U_k]^*.
\end{aligned}
\end{equation}
Using this along with \eqref{eqn:bdy1} and \eqref{eqn:epeqn} we have:
\begin{equation} \label{eqn:bumean}
\begin{aligned}
&B\underline{U_0} = 0 \quad
B\underline{U_1} = \uG\\
&B \underline{U_2} = 0 \quad
B \underline{U_j}= 0 \text{ for } j\geq 2
\end{aligned}
\end{equation}
\begin{equation} \label{eqn:bucheck}
\begin{aligned}
B\ccU_0 &= -BE_{P_1}U_0 \\
B\ccU_1 &= G^* - BE_{P_1}U_1 + BR[L(\D)U_0 + f(\theta_3)MU_0]^*\\
B\ccU_2 &= - BE_{P_1}U_2 + BR[L(\D)U_1 + f(\theta_3)MU_1]^*\\
B\ccU_j &= - BE_{P_1}U_j + BR[L(\D)U_{j-1} + f(\theta_3)MU_{j-1}]^* \text{ for } j\geq 2.
\end{aligned}
\end{equation}
Using \eqref{eqn:bdot1}, we have:
\begin{equation} \label{eqn:bdot}
\begin{aligned}
& b \cdot [G^* - BE_{P_1}U_1 + BR(L(\D)U_0 + f(\theta_3)MU_0)^*] = 0\\
& b \cdot [- BE_{P_1}U_2 + BR(L(\D)U_1 + f(\theta_3)MU_1)^*] = 0\\
& b \cdot [- BE_{P_1}U_j + BR(L(\D)U_{j-1} + f(\theta_3)MU_{j-1})^*] = 0 \text{ for } j\geq 2
\end{aligned}
\end{equation}
\subsection{Determining Solutions}\hfill\\

We write as before:
\begin{equation}
U_k = \underline{U_k} + \sigma_k^1r_1 + \sigma_k^2r_2 + \sigma_k^3r_3 + (I-E_P)U_k.
\end{equation}

1. Determining the leading profile $U_0$\\

From \eqref{eqn:epeqn}, we have:
\begin{equation}
U_0 = \underline{U_0} + \sigma_0^1(t,x,\theta_1)r_1+\sigma_0^2(t,x,\theta_2)r_2 + \sigma_0^3(t,x,\theta_3)r_3
\end{equation}
where $\sigma_0^m(t,x,\theta_m)$ are scalars.\\
The first equation \eqref{eqn:eqeqn} gives rise to four equations. We first consider the mean of both sides:
\begin{equation}
\begin{aligned}
&E_{Q_0}[L(\D)U_0+f(\theta_3)MU_0] = 0 \\
\Leftrightarrow &L(\D)\underline{U_0}+\underline{f(\theta_3)M\sigma_0^3(t,x,\theta_3)r_3} = 0.
\end{aligned}
\end{equation}
Using the boundary condition given by \eqref{eqn:bumean} and the fact $Mr_3 = 0$, we have:
\begin{equation}
\begin{aligned}
&L(\D)\underline{U_0} = 0\\
&B\underline{U_0} = 0.
\end{aligned}
\end{equation}
This gives $\underline{U_0} = 0$.\\

We then take the pure $\theta_1$ mode on both sides of the first equation of \eqref{eqn:eqeqn}:
\begin{equation}
\begin{aligned}
E_{Q_1}[L(\D)U_0+f(\theta_3)MU_0] = 0. 
\end{aligned}
\end{equation}
Notice that the only oscillation of $L(\D)U_0+f(\theta_3)MU_0$ in $\theta_1$ is $\sigma_0^1$ since there is no resonance, thus we have:
\begin{equation}
E_{Q_1}L(\D)E_{P_1}U_0 = 0.
\end{equation}
Using proposition \ref{propX}, we have $X_{\phi_1}\sigma_0^1B_2r_1 = 0$. Since $\phi_1$ is outgoing, and $\sigma_0^1 = 0$ when $t < 0$, we have $\sigma_0^1 = 0$.\\

Now we take a moment to discuss implications of $\sigma_0^1 = 0$. We claim that $\sigma_0^1 = 0$ implies $E_{P_1}U_j = 0$ for all $j$. First, we consider $U_1$: 
\begin{equation}
U_1 = \underline{U_1} + \sigma_1^1r_1 + \sigma_1^2r_2 + \sigma_1^3r_3 + (I-E_P)U_1
\end{equation}
where $(I-E_P)U_1 = -R[L(\D)U_0+f(\theta_3)MU_0]$. Plugging this equation in \eqref{eqn:eqeqn}, we have:
\begin{equation} \label{eqn:equ1f}
\begin{aligned}
0 &= E_Q[L(\D)U_1+f(\theta_3)MU_1] = E_Q[L(\D)(E_PU_1 + (I-E_P)U_1) + f(\theta_3)M(E_PU_1+ (I-E_P)U_1)] \\
&\Leftrightarrow 
E_Q[L(\D)(E_PU_1)+ f(\theta_3)ME_PU_1] = -E_Q[L(\D)(I-E_P)U_1+f(\theta_3)M(I-E_P)U_1].
\end{aligned}
\end{equation}
We take the pure $\theta_1$ oscillation of \eqref{eqn:equ1f}. Since there is no resonance, all the terms involving $f(\theta_3)$ is zero and on the right hand side only the terms involving $U_0$ is left. We also know $U_0$ has no $\theta_1$ mode, so we obtain:
\begin{equation} \label{u1theta1}
E_{Q_1}L(\D)E_{P_1}U_{1} = X_{\phi_1}\sigma_1^1B_2r_1 = 0.
\end{equation}
Therefore, $X_{\phi_1}$ is outgoing and $\sigma_1^1 = 0$ in $t<0$ $\Rightarrow \sigma_1^1 = 0$. 

We do the above calculation for other profiles, same calculation shows that $U_{j}$ has no $\theta_1$ mode if $U_{j-1}$ has no $\theta_1$ mode. Therefore, $E_{P_1}U_j = 0$ for all $j$. We will use this observation to continue our determining the solutions.\\

We continue to solve the rest of $U_0$. The other two equations given by \eqref{eqn:eqeqn} is :
\begin{equation}
\begin{aligned}
E_{Q_2}[L(\D)U_0+f(\theta_3)MU_0] = 0 \\
E_{Q_3}[L(\D)U_0+f(\theta_3)MU_0] = 0. \\
\end{aligned}
\end{equation}
Using proposition \ref{propX} and no resonance, we have:
\begin{equation} \label{eqn: u0sigmaint}
\begin{aligned}
&X_{\phi_2}\sigma_0^2B_2r_2 = 0 \\
&X_{\phi_3}\sigma_0^3B_2r_3 + Q^3[f(\theta_3)M\underline{U_0} + f(\theta_3)M\sigma_0^3r_3]= 0. \\
\end{aligned}
\end{equation}
Knowing $\underline{U_0} = 0$ and $Mr_3 = 0$, the second equation become $X_{\phi_3}\sigma_0^3B_2r_3 = 0$. To solve these transport equations, we need a boundary condition for $\sigma_0^2, \sigma_0^3$. To summarize, we now have $U_0 = \sigma_0^2r_2 + \sigma_0^3r_3$.

On the boundary, we write:
\begin{equation}
U_0 = \sigma_0^2(t,x_1,0,\theta_0)r_2 + \sigma_0^3(t,x_1,0,\theta_0)r_3 = a_0(t,x_1,\theta_0)e +  \check{a}_0(t,x_1,\theta_0)\check{e} 
\end{equation}
where $e = e_2 + e_3$,and $\check{e} = \check{e}_2 + \check{e}_3$ for $e_m, \check{e}_m \in $span $r_m$. From \eqref{eqn:bucheck}, we have: $B\ccU = 0$, thus $\check{a}_0 = 0$. 

 It is left to to determine $a_0$. We apply $\D_{\theta_0}$ to both sides of the first equation in \eqref{eqn:bdot}. We already know that $EP_1U_1 = 0$. We have\footnote{Whenever the boundary equation involves the operator $R$ like this, the evaluation at boundary  is taken after applying $R$.}:
\begin{equation} \label{eqn:u0bry}
\begin{aligned}
\D_{\theta_0}b\cdot G^* + \D_{\theta_0}b\cdot BR[L(\D)U_0
+f(\theta_3)MU_0] &=0 \\
\Leftrightarrow \ \D_{\theta_0}b\cdot BR[L(\D)(a_0(t,x_1,\theta_2)e_2 + a_0(t,x_1,\theta_3)e_3))\\ \hspace{1in} +f(\theta_3)M(a_0(t,x_1,\theta_2)e_2 + a_0(t,x_1,\theta_3)e_3))] &= \D_{\theta_0}b\cdot G^*.
\end{aligned}
\end{equation}
Knowing the right hand side, we can solve this equation by the discussion in section \ref{abae}, which gives the unique solution $a_0(t,x_1,\theta_0)$. Now we know $\sigma_0^2(t,x_1,0,\theta_0),$ and $ \sigma_0^3(t,x_1,0,\theta_0)$. Putting these boundary conditions with interior equations \eqref{eqn: u0sigmaint}, we have 
\begin{equation}
\begin{aligned}
\begin{cases}
&X_{\phi_2}\sigma_0^2B_2r_2 = 0 \\
&\sigma_0^2(t,x_1,0,\theta_0) = a_0(t,x_1,\theta_0)e_2+\check{a}_0(t,x_1,\theta_0)\check{e}_2 = a_0(t,x_1,\theta_0)e_2
\end{cases}
\\
\begin{cases}
&X_{\phi_3}\sigma_0^3B_2r_3 = 0 \\
&\sigma_0^3(t,x_1,0,\theta_0) = a_0(t,x_1,\theta_0)e_3+\check{a}_0(t,x_1,\theta_0)\check{e}_3 = a_0(t,x_1,\theta_0)e_3.
\end{cases}
\end{aligned}
\end{equation}
These equations uniquely determined $\sigma_0^2, \sigma_0^3$, thus we now have $U_0$. \\

2. Determining $U_1$:\\

We write 
\begin{equation}
U_1 = \underline{U_1} + \sigma_1^1r_1 + \sigma_1^2r_2 + \sigma_1^3r_3 + (I-E_P)U_1
\end{equation}
where $(I-E_P)U_1 = -R[L(\D)U_0+f(\theta_3)MU_0]$.

We plug this equation in \eqref{eqn:eqeqn}, we have:
\begin{equation} \label{eqn:equ1}
\begin{aligned}
0 &= E_Q[L(\D)U_1+f(\theta_3)MU_1] = E_Q[L(\D)(E_PU_1 + (I-E_P)U_1) + f(\theta_3)M(E_PU_1+ (I-E_P)U_1)] \\
&\Leftrightarrow 
E_Q[L(\D)(E_PU_1)+ f(\theta_3)ME_PU_1] = -E_Q[L(\D)(I-E_P)U_1+f(\theta_3)M(I-E_P)U_1].
\end{aligned}
\end{equation}
Taking the mean of this equation together with the boundary condition gives $U_1$: 
\begin{equation}
\begin{aligned}
&L(\D)\underline{U_1} = 0 \\
&B\underline{U_1} = \underline{G}.
\end{aligned}
\end{equation}
This gives $\underline{U_1}$. 

Taking the $\theta_2, \theta_3$ mode of \eqref{eqn:equ1} yields equations:
\begin{equation}
	\begin{aligned}
	&E_{Q_2}L(\D)(E_{P_2}U_1) = -E_{Q_2}[L(\D)(I-E_P)U_1+f(\theta_3)M(I-E_P)U_1] \\
	&E_{Q_3}[L(\D)(E_{P_3}U_1)+ f(\theta_3)M(\underline{U_1} + \sigma_3^3r_3)] = -E_{Q_3}[L(\D)(I-E_P)U_1+f(\theta_3)M(I-E_P)U_1].
	\end{aligned}
\end{equation}
Using $Mr_3 = 0$ and moving the term involving $\underline{U_1}$ to the right hand side, we have:
\begin{equation} \label{eqn:u123transport}
	\begin{aligned} 
	&X_{\phi_2}\sigma_1^2B_2r_2 = -E_{Q_2}[L(\D)(I-E_P)U_1+f(\theta_3)M(I-E_P)U_1] \\
	&X_{\phi_3}\sigma_1^3B_2r_3 = -E_{Q_3}[L(\D)(I-E_P)U_1+f(\theta_3)M(I-E_P)U_1 - f(\theta_3)M\underline{U_1}].
	\end{aligned}
\end{equation}
These are inhomogeneous transport equations. Knowing $\sigma_1^2, \sigma_1^3 = 0$ in $t < 0$, we only need a boundary condition. We have on the boundary:
\begin{equation}
E_{P_{in}}U_1 = \sigma_1^2(t,x_1,0,\theta_0)r_2 + \sigma_1^3(t,x_1,0,\theta_0)r_3 = a_1(t,x_1,\theta_0)e +  \check{a}_1(t,x_1,\theta_0)\check{e}.
\end{equation}
The equation from \eqref{eqn:bucheck} $B\ccU_1 = G^* - BE_{P_1}U_1 + BR[L(\D)U_0 + f(\theta_3)MU_0]^*$ yields $\ccU_1$ since the right hand side is known. It is left to determine $a_1$. We apply $\D_{\theta_0}$ to both sides of \eqref{eqn:bdot}, we have:
\begin{equation}
\begin{aligned}
\D_{\theta_0} b \cdot [- BE_{P_1}U_2 + BR(L(\D)U_1 + f(\theta_3)MU_1)^*] &= 0\\
\Leftrightarrow
\D_{\theta_0} b \cdot [ BR(L(\D)(a_1(t,x_1,\theta_2)e_2 + a_0(t,x_1,\theta_3)e_3) + 
\\ \hspace{1in}
f(\theta_3)M(a_1(t,x_1,\theta_2)e_2 + a_0(t,x_1,\theta_3)e_3))^*] &= h_1
\end{aligned}
\end{equation}
where the second line is obtained by using $BE_{P_1}U_2 = 0$, moving $\ccU_1, \underline{U_1}, (I-E_P)U_1$ terms to the right hand side, and letting $h_1$ denotes the known right hand side. This equation yields $a_1$ by the discussion in section \ref{abae}. Thus we have now determined $\sigma_1^2, \sigma_1^3$ from the following equations:
\begin{equation}
\begin{aligned}
\begin{cases}
&X_{\phi_2}\sigma_1^2B_2r_2 = H_1^2 \\
&\sigma_1^2(t,x_1,0,\theta_0) = a_1(t,x_1,\theta_0)e_2+\check{a}_1(t,x_1,\theta_0)\check{e}_2
\end{cases}
\\
\begin{cases}
&X_{\phi_3}\sigma_1^3B_2r_3 = H_1^3 \\
&\sigma_1^3(t,x_1,0,\theta_0) = a_1(t,x_1,\theta_0)e_3+\check{a}_1(t,x_1,\theta_0)\check{e}_3
\end{cases}
\end{aligned}
\end{equation}
where $H_1^m$, $m =2,3$ denotes the known right hand side from \eqref{eqn:u123transport}. Now we have determined $U_1$. The pattern to determine the rest of the $U_k$ is now clear.

\subsection{Analysis of the boundary equations} \label{abae}
\hfill \\

\textbf{The Leading Boundary Equation}\\

This section is a similar calculation as section 5.5 \cite{W4} which gives solutions to equations \eqref{eqn:u0bry}. \\
We first analyze the boundary equation \eqref{eqn:u0bry} of leading profile $U_0$:
\begin{equation} \label{u01}
\D_{\theta_0}b\cdot BR[L(\D)(a_0(t,x_1,\theta_2)e_2 + a_0(t,x_1,\theta_3)e_3)+f(\theta_3)M(a_0(t,x_1,\theta_2)e_2 + a_0(t,x_1,\theta_3)e_3)] = \D_{\theta_0}b\cdot G^* 
\end{equation}
Our goal is to show that \eqref{u01} takes the form:
\begin{equation}
X_{Lop}a(t,x_1,\theta_0) + (e^{i\theta_0}m_1(D_{\theta_0})+ e^{-i\theta_0}m_2(D_{\theta_0}))a(t,x_1,\theta_0) = h(t,x_1,\theta_0), \ h=0 \text{ in } t < 0, 
\end{equation}
where $m_i(D_{\theta_0})$ is a bounded Fourier multiplier and $X_{Lop}$ is \footnote{$X_{Lop}$ is in fact a characteristic vector field of the Lopatinski determinant \cite{CG10}.}
\begin{equation}
X_{Lop} = c_0\D_t + c_1\D_{x_1} \text{ , with } c_j \in \RR, c_0 \neq 0
\end{equation}
for some $c_0, c_1$. We start by consider the first term in the left hand side of \eqref{u01}. Using $R_mB_2P_m = 0$ for $m=2,3$, we notice that the $B_2$ term in $L(\D)$ vanished after applying $R$. Thus, we have:
\begin{equation}
\D_{\theta_0}b\cdot BRL(\D)(a_0(t,x_1,\theta_2)e_2 + a_0(t,x_1,\theta_3)e_3) = \D_{\theta_0}b\cdot BRL'(\D)(a_0(t,x_1,\theta_2)e_2 + a_0(t,x_1,\theta_3)e_3)
\end{equation}
for $L'(\D) = \D_t + B_1\D_{x_1}$.\\
It is shown in \cite{CG10} that 
\begin{equation}
\D_{\theta_0}b\cdot BRL'(\D)(a(t,x,\theta_2)e_2+a(t,x,\theta_3)e_3) = X_{Lop}a(t,x,\theta_0).
\end{equation}
Thus, \eqref {u01} becomes:
\begin{equation}
X_{Lop}a(t,x,\theta_0) + \D_{\theta_0}b\cdot BR(e^{i\theta_3}+ e^{-i\theta_3})M(a_0(t,x_1,\theta_2)e_2 + a_0(t,x_1,\theta_3)e_3) = \D_{\theta_0}b\cdot G^* .
\end{equation}
Using $Me_3 = 0$, we have $\D_{\theta_0}b\cdot BR(e^{i\theta_3}+ e^{-i\theta_3})Ma_0(t,x,\theta_3)e_3 = 0$.\\
We write $a(t,x_1,\theta_0) = \sum a_{0,k}(t,x_1)e^{ik\theta_0}$ where the sum is taken over all integers, the second term of \eqref{u01}:
\begin{equation} \label{eqn:uo2}
\begin{aligned}
&\D_{\theta_0}b\cdot BR(e^{i\theta_3}+ e^{-i\theta_3})Ma_0(t,x,\theta_2)e_2 = 
\D_{\theta_0}b\cdot BRMe_2 (\sum a_{0,k}(t,x_1)(e^{ik\theta_2+i\theta_3}+e^{ik\theta_2-i\theta_3})\\
&= \D_{\theta_0} [b\cdot B \sum_{k \in \ZZ} (L(ikd\phi_2+id\phi_3)^{-1}Me_2)a_{0,k}e^{i(k+1)\theta_0} \\& \hspace{1in} +  b\cdot B \sum_{k \in \ZZ} (L(ikd\phi_2-id\phi_3)^{-1}Me_2)a_{0,k}e^{i(k-1)\theta_0}]\\
&= b\cdot B \sum_{k \in \ZZ} (L(ikd\phi_2+id\phi_3)^{-1}Me_2)(k+1)a_{0,k}e^{i(k+1)\theta_0} \\& \hspace{1in} +b\cdot B \sum_{k \in \ZZ} (L(ikd\phi_2-id\phi_3)^{-1}Me_2)(k-1)a_{0,k}e^{i(k-1)\theta_0}
\end{aligned}
\end{equation}
We need the following lemma whose proof is easy.
\begin{lem} \label{Rlem} \cite{W4}
For $(k,l) \in \mathcal{N}$ and for any $X \in \CC^3$ 
\begin{align}
&L^{-1}(kd\phi_2+ld\phi_3)X = c_1r_1 + c_2r_2 + c_3r_3 \text{, where } \\
&c_1 = \frac{\ell_1X}{k(\omega_2-\omega_1)+l(\omega_3-\omega_1)}\quad
c_2 = \frac{\ell_2X}{l(\omega_3-\omega_2)}\quad
c_3 = \frac{\ell_3X}{k(\omega_2-\omega_3)}
\end{align}
\end{lem}
Using above lemma \ref{Rlem} and $b \cdot B r_p = 0$ for $p = 2,3$, we continue the string of equations in \eqref{eqn:uo2}:
\begin{equation}
	\begin{aligned}
	\eqref{eqn:uo2} &=  b\cdot B \sum_{k\in \ZZ} \frac{\ell_1Me_2(k+1)a_{0,k}e^{i(k+1)\theta_0}r_1}{k(\omega_2-\omega_1)+ (\omega_3-\omega_1)} +
	 b\cdot B \sum_{k\in \ZZ} \frac{\ell_1Me_2(k-1)a_{0,k}e^{i(k-1)\theta_0}r_1}{k(\omega_2-\omega_1)- (\omega_3-\omega_1)} 
	\\& = (b\cdot Br_1) (\ell_1Me_2)\{\sum_{k \in \ZZ} \frac{(k+1)a_{0,k}e^{i(k+1)\theta_0}}{k(\omega_2-\omega_1)+ (\omega_3-\omega_1)} + \sum_{k \in \ZZ} \frac{(k-1)a_{0,k}e^{i(k-1)\theta_0}}{k(\omega_2-\omega_1)- (\omega_3-\omega_1)}\}\\
	& = (b\cdot Br_1) (\ell_1Me_2)\{e^{i\theta_0}\sum_{k \in \ZZ} \frac{(k+1)a_{0,k}e^{ik\theta_0}}{k(\omega_2-\omega_1)+ (\omega_3-\omega_1)} + e^{-i\theta_0}\sum_{k \in \ZZ} \frac{(k-1)a_{0,k}e^{ik\theta_0}}{k(\omega_2-\omega_1)- (\omega_3-\omega_1)}\}
	\end{aligned}
\end{equation}
since $Br_p \in BE^s(\beta_l) \Rightarrow b\cdot Br_p = 0$ for $ p= 2,3$. We observe that the set $\{\frac{k \pm 1}{k(\omega_2-\omega_1)\pm (\omega_3-\omega_1)}, k\in \ZZ\}$ is bounded.

Now, the equation \ref{u01} takes the form:
\begin{equation}
X_{Lop}a_0(t,x_1,\theta_0) + (e^{i\theta_0}m_1(D_{\theta_0})+ e^{-i\theta_0}m_2(D_{\theta_0}))a_0(t,x_1,\theta_0) = h(t,x_1,\theta_0), \ h=0 \text{ in } t < 0, 
\end{equation}
where $m_i(D_{\theta_0})$ is the bounded Fourier multiplier can be read off from above calculation , and $h(t,x_1,\theta_0)$ is known. An argument based on energy estimate yields a unique solution $a_0(t,x_1,\theta_0) \in H^\infty((-\infty,T]\times \RR \times \TT)$ satisfying $a = 0$ in $t < 0$. \\

\textbf{The General Boundary Equations} \\

Now we analyze the general boundary equations for all $U_j$ that gives boundary conditions for all $\sigma_j^2 = a_j(t,x_1,\theta_2)e + \check{a}_j(t,x_1,\theta_2)\check{e},\  \sigma_j^3 = a_j(t,x_1,\theta_3)e + \check{a}_j(t,x_1,\theta_3)\check{e}$. \\
Applying $\D_{\theta_0}$ to both sides of \eqref{eqn:bdot}, we get a equation in the form.
\begin{equation} \label{eqn:2.42}
\D_{\theta_0} b\cdot (-BE_{P_1}U_{j+1}+B[RL(\D)U_j+f(\theta_3)MU_j])=h_j(t,x,\theta_0)
\end{equation}
At this point of constructing profiles, we have already determined $E_{P_1}U_{j} = 0, E_{P_1}U_{j+1} = 0$ and $(I-E_P)U_j$. We also note that from \ref{eqn:bucheck}, we know $\ccU$. We move all the known terms to the right and drop all the subscript for simpler notation. Recall that on the boundary we write $E_{P_{in}}U(t,x_1,0,\theta_0) = \sigma^2(t,x_1,0,\theta_0)r_2 + \sigma^3(t,x_1,0,\theta_0)r_3 = a(t,x_1,\theta_0)e + \check{a}(t,x_1,\theta_0)\check{e}$, equation \eqref{eqn:2.42}  becomes:
\begin{equation}
\D_{\theta_0} b\cdot BR[L(\D)(a(t,x_1,\theta_2)e_2 +a(t,x_1,\theta_3)e_3 )+f(\theta_3)M(a(t,x_1,\theta_2)e_2 +a(t,x_1,\theta_3)e_3 )] = h(t,x,\theta_0)
\end{equation} 
This equation is identical to \eqref{u01} that we've discussed with a different known right hand side. The same discussion yields $a(t,x_1,\theta_0) \in H^\infty ((-\infty,T]\times \RR \times \TT)$.\\

\newpage
\section{Bibliography}
\bibliographystyle{alpha}
\bibliography{bib}

\end{document}